\documentclass[11pt,english]{amsart}
\usepackage[T1]{fontenc}
\usepackage[latin9]{inputenc}
\usepackage[a4paper,margin=3cm]{geometry}
\usepackage{mathrsfs}
\usepackage{bm}
\usepackage{bbm}
\usepackage{amstext}
\usepackage{amsthm}
\usepackage{amssymb}
\usepackage{graphicx}
\usepackage{appendix}
\usepackage[authoryear]{natbib}
\usepackage{babel}
\usepackage{mathrsfs}

\makeatletter
\numberwithin{equation}{section}
\numberwithin{figure}{section}
\theoremstyle{definition}
\newtheorem{example}{\protect\examplename}[section]
\theoremstyle{plain}
\newtheorem{assumption}{\protect\assumptionname}
\theoremstyle{plain}
\newtheorem{thm}{\protect\theoremname}[section]
\theoremstyle{remark}
\newtheorem{rem}{\protect\remarkname}[section]
\theoremstyle{plain}
\newtheorem{lem}{\protect\lemmaname}[section]
\theoremstyle{plain}

\theoremstyle{plain}
\newtheorem{cor}{\protect\corollaryname}[section]

\setcitestyle{authorname, round}

\usepackage{amsfonts}
\usepackage{algorithm}
\usepackage{algorithmicx}
\usepackage{algpseudocode}
\usepackage{float}

\usepackage[colorlinks, allcolors=blue]{hyperref}

\newcommand{\md}{\mathop{}\mathopen\mathrm{d}}
\newcommand{\me}{\mathop{}\mathopen\mathrm{e}}

\newcommand{\bR}{\mathbb{R}}
\newcommand{\R}{\mathbb{R}}

\newcommand{\bE}{\mathbb{E}}
\newcommand{\E}{\mathbb{E}}

\newcommand{\cL}{\mathcal{L}}

\newcommand{\sP}{\mathscr{P}}

\newcommand{\cW}{\mathcal{W}}
\newcommand{\cK}{\mathcal{K}}

\newcommand{\eps}{\varepsilon}

\newcommand{\err}{\theta_{n}}
\newcommand{\tr}{\intercal}

\newcommand{\id}{\mathrm{Id}}

\makeatother

\providecommand{\assumptionname}{Assumption}
\providecommand{\corollaryname}{Corollary}
\providecommand{\examplename}{Example}
\providecommand{\lemmaname}{Lemma}
\providecommand{\remarkname}{Remark}
\providecommand{\theoremname}{Theorem}

\begin{document}
\title{Sequential propagation of chaos}
\author{Kai Du}
\address{(K. Du) Shanghai Center for Mathematical Sciences, Fudan University,
Shanghai, China; Shanghai Artificial Intelligence Laboratory, 701
Yunjin Road, Shanghai, China}
\email{kdu@fudan.edu.cn}
\author{Yifan Jiang}
\address{(Y. Jiang) Mathematical Institute, University of Oxford, Oxford, UK}
\email{yifan.jiang@maths.ox.ac.uk}
\author{Xiaochen Li}
\address{(X. Li) School of Mathematical Sciences, Fudan University, Shanghai,
China}
\email{xcli18@fudan.edu.cn}
\keywords{propagation of chaos, particle system, sequential interaction, McKean--Vlasov
process, convergence rate}
\begin{abstract}
\noindent
A new class of particle systems with sequential interaction is proposed to approximate the McKean--Vlasov process that originally arises as the limit of the mean-field interacting particle system.
The weighted empirical measure of this particle system is proved to converge to the law of the McKean--Vlasov process as the system grows.
Based on the Wasserstein metric, quantitative propagation of chaos results are obtained for two cases: the finite time estimates under the monotonicity condition and the uniform in time estimates under the dissipation and the non-degenerate conditions.
Numerical experiments are implemented to demonstrate the theoretical results.
\end{abstract}

\maketitle

\section{Introduction}

This paper studies the asymptotic behavior of a particle system with sequential interaction. 
The major feature of the system is that each particle  is influenced only by the particles with smaller ordinal numbers. 
Specifically, the $n$-th particle process
$X^{n}$ (\(n\geq2\)) is determined recursively by
\begin{equation}
\left\{
\begin{aligned}\md X_{t}^{n} & =b(t,X_{t}^{n},\mu_{t}^{n-1})\md t+\sigma(t,X_{t}^{n},\mu_{t}^{n-1})\md W_{t}^{n};\\
\mu_{t}^{n} & =\mu_{t}^{n-1}+\alpha_{n}(\delta_{X_{t}^{n}}-\mu_{t}^{n-1}),
\end{aligned}
\right.
\label{eq:ADPS}
\end{equation}
with \(X^{1}_{t}\equiv X^{1}_{0}\) and \(\mu_{t}^{1}=\delta_{X_{t}^{1}}\).
Here, the update rate \(\{\alpha_{n}\}_{n\geq 1}\) is a decreasing positive sequence with \(\alpha_{1}=1\), $W^{n}$ are independent multidimensional Brownian motions, and the initial data $X_{0}^{n}\sim\mu_{0}$ are i.i.d. $\R^{d}$-valued random variables independent of $\{W^{n}\}_{n\geq 1}$.
Evidently, $\mu_{t}^{n}$ is a weighted empirical measure of $X_{t}^{1},\dots,X_{t}^{n}$; if  we take $\alpha_{n}=1/n$, $\mu_{t}^{n}$ is the classical empirical measure.

As the interaction is asymmetric and heterogeneous, the model \eqref{eq:ADPS} is in sharp contrast to the corresponding
 mean-field interacting system $\{X^{n,N}:n=1,\dots,N\}$
given by 
\begin{equation}
    \label{eq:mean-field}
\left\{\begin{aligned}
    \md X_{t}^{n,N}&=b(t,X_{t}^{n,N},\mu_{t}^{N})\md t+\sigma(t,X_{t}^{n,N},\mu_{t}^{N})\md W_{t}^{n},\\
    \mu_{t}^{N}&=\frac{1}{N}\sum_{n=1}^{N}\delta_{X^{n,N}_{t}}.
\end{aligned}\right.\end{equation}
It is known from the theory of propagation of chaos (PoC for short,
cf. \citet{mckean1967propagation,sznitman1991topics}) that, as $N$
increases, the empirical measure of the system \eqref{eq:mean-field}
may converge to the law of a McKean--Vlasov process described by
\begin{equation}\left\{
\begin{aligned}\md X_{t} & =b(t,X_{t},\mu_{t})\md t+\sigma(t,X_{t},\mu_{t})\md W_{t},\\
\mu_{t} & =\cL(X_{t}):=\mathrm{Law}(X_{t}).
\end{aligned}\right.
\label{eq:McKean--Vlasov}
\end{equation}
This paper then obtains a similar property for the system \eqref{eq:ADPS}:
as $n$ tends to infinity, the weighted empirical measure $\mu_{t}^{n}$
as well as the law of $X_{t}^{n}$ converges to $\mu_{t}$ defined
in \eqref{eq:McKean--Vlasov} in a very general setting.

Our primary motivation of introducing the particle system \eqref{eq:ADPS} is to construct a computational friendly approximation to McKean--Vlasov processes and the associated nonlinear Fokker--Planck--Kolmogorov equations (cf. \citet{frank2005nonlinear}). 
The McKean--Vlasov process was initially proposed by \citet{mckean1966class} to give a probabilistic interpretation for nonlinear Vlasov equations, and has found a wide range of applications (see \citet{carmona2018probabilistic,chaintron2021propagation} and references therein).  
Although the PoC property is widely applied in computational problems for McKean--Vlasov processes (see \citet{bossy1997stochastic,bao2021first,dos2022simulation}, etc.), the system (\ref{eq:mean-field}) itself is fully coupled and $N$-dependent.
To achieve certain numerical accuracy, a sufficiently large $N$ has to be set a priori and the change of \(N\) will lead to a recomputation of the whole particle system.
On the other hand, thanks to its recursive form, the system \eqref{eq:ADPS} can be simply computed particle-by-particle (or batch-by-batch), resulting in a great reduction of the computational burden. 
More significantly, as more and more particles are added (without affecting existing particles), the approximation accuracy is continuously improved until reaching a desired level; in other words, the required number of particles no longer needs to be specified in advance. 

Particle systems with heterogeneous interaction have attracted a growing interest in recent years. 
The classical mean-field framework in which the particles are interacted with each other has been extended to allow interactions described by general networks. 
In most models, the interactions are encoded in a graph sequence, i.e., two particles are interacting if and only if they are connected in the underlying graph (see \citet{delarue2017mean,bhamidi2019weakly,bayraktar2020graphon,jabin2021mean} and references therein). 
We remark that, in those models, the degree of each particle (as a node in the graph) tends to infinity as the system size $N$ increases. 
This is essentially different from our model in which the degree of each particle is fixed and independent of the system size. 
Some works consider CDF-based or position-based interacting particle systems (see \citet{jourdain2008propagation,andreis2019ergodicity} etc.), where the interactions are time-varying. 
To our best knowledge, it seems to be no discussion on particle systems like \eqref{eq:ADPS} in the literature. 
Apart from the computational consideration, the model \eqref{eq:ADPS} and its extensions are expected to have applications in the study of complex networks (cf. \citet{arenas2008synchronization}), which are left to the future work.

The main results of this paper give quantitative PoC estimates for the system \eqref{eq:ADPS} in two cases: finite time estimates under the monotonicity condition (see Theorem \ref{thm:APoC-W}), and the uniform in time estimates under the dissipation and the non-degenerate conditions (see Theorem \ref{thm:time uniform-1}). 
The difference between measures is quantified by the Wasserstein distance. 
The proofs make use of coupling arguments: we adopt McKean's argument of synchronous coupling (cf. \citet{mckean1967propagation,sznitman1991topics}) in the finite time case, and for uniform in time estimates, we apply the technique of reflection coupling that was developed in \citet{Eberle2016Quantitative,Luo_W_p,uniform,Liu_2021}.
The new challenge is caused by the special interaction mechanism of our model. 
To see this, let $x_{n}$ denote the quantity to be estimated for the $n$-th particle. 
In most existing models (usually an $N$-particle system), one may obtain a relation of the form $x_{n}\le f(x_{1},\dots,x_{N})$, then with the help of certain structural features like exchangeability or similarity one can derive some upper bound for $x_{n}$ or some combination of $x_{1},\dots,x_{N}$. 
However, the corresponding relation in our model has a recursive form, say, $x_{n}\le f_{n}(x_{1},\dots,x_{n})$, which cannot be handled analogously. 
For this, new estimates for recursive inequalities are proved (see Subsection~\ref{subsec:Estimates-for-recursive}).
Moreover, we derive a new estimate for the convergence rate of weighted empirical measures of an i.i.d. sequence in the Wasserstein distance (see Lemma \ref{lemma:iid}). 
The results for classical empirical measures are rich in the literature, see, for example, \citet{horowitz1994mean,fournier2015rate,ambrosio2019pde} and references therein.

Although our discussion focuses on \eqref{eq:ADPS}, the results are still valid for some similar models by analogous arguments. 
For instance, we can consider the model where the $n$-th particle process is governed by 
\[\left\{
\begin{aligned}\md X_{t}^{n} & =b(t,X_{t}^{n},\mu_{t}^{n})\md t+\sigma(t,X_{t}^{n},\mu_{t}^{n})\md W_{t}^{n},\\
\mu_{t}^{n} & =\mu_{t}^{n-1}+\alpha_{n}(\delta_{X_{t}^{n}}-\mu_{t}^{n-1});
\end{aligned}\right.
\]
comparing with \eqref{eq:ADPS}, the measure $\mu_{t}^{n}$ takes the place of $\mu_{t}^{n-1}$ in the equation for $X^{n}$. 
Also, one can add the particles batch-by-batch rather than particle-by-particle:
the $n$-th batch $\{X^{i,n}:i=1,\dots,n\}$ can be constructed as
\[\left\{
\begin{aligned}\md X_{t}^{i,n} & =b(t,X_{t}^{i,n},\mu_{t}^{n})\md t+\sigma(t,X_{t}^{i,n},\mu_{t}^{n})\md W_{t}^{i,n},\\
\mu_{t}^{n} & =\mu_{t}^{n-1}+\alpha_{n}\bigg(\frac{1}{n}\sum_{i=1}^{n}\delta_{X_{t}^{i,n}}-\mu_{t}^{n-1}\bigg).
\end{aligned}\right.
\]
The batch model could be more efficient in numerical computation.

Remarkably, our model exhibits a comparable convergence rate with the classical mean-filed system, but with less computation than the latter. 
This is not just indicated by our theoretical results but also demonstrated by numerical experiments. 

\begin{example}[Curie--Weiss mean-field lattice model]
Consider the equation 
\[
\md X_{t}=[-\beta(X_{t}^{3}-X_{t})+\beta K\bE X_{t}]\md t+\sigma\md W_{t},\quad X_{0}=1.
\]
We take $\beta=1$, $K=1/2$, $T=3$ and $\sigma=1$, and numerically compare the performance of the classical PoC method and our sequential PoC method. 
Figure \ref{fig:Curie-Weiss-mean-field lattice model} (left) shows the comparison of convergence rates between these two methods. 
In addition, we draw a reference line with slope $-1/2$ (in the log-log diagram) to indicate the convergence rate. 
Figure \ref{fig:Curie-Weiss-mean-field lattice model} (right) shows the approximate density functions simulated by the sequential PoC with different numbers of particles $N$.
As $N$ increases, the approximate density function is gradually updated, and no longer changes drastically when $N$ runs over $10^5$, resulting in a fine approximation to the true density.
\begin{figure}[htbp]
\includegraphics[width=0.5\textwidth]{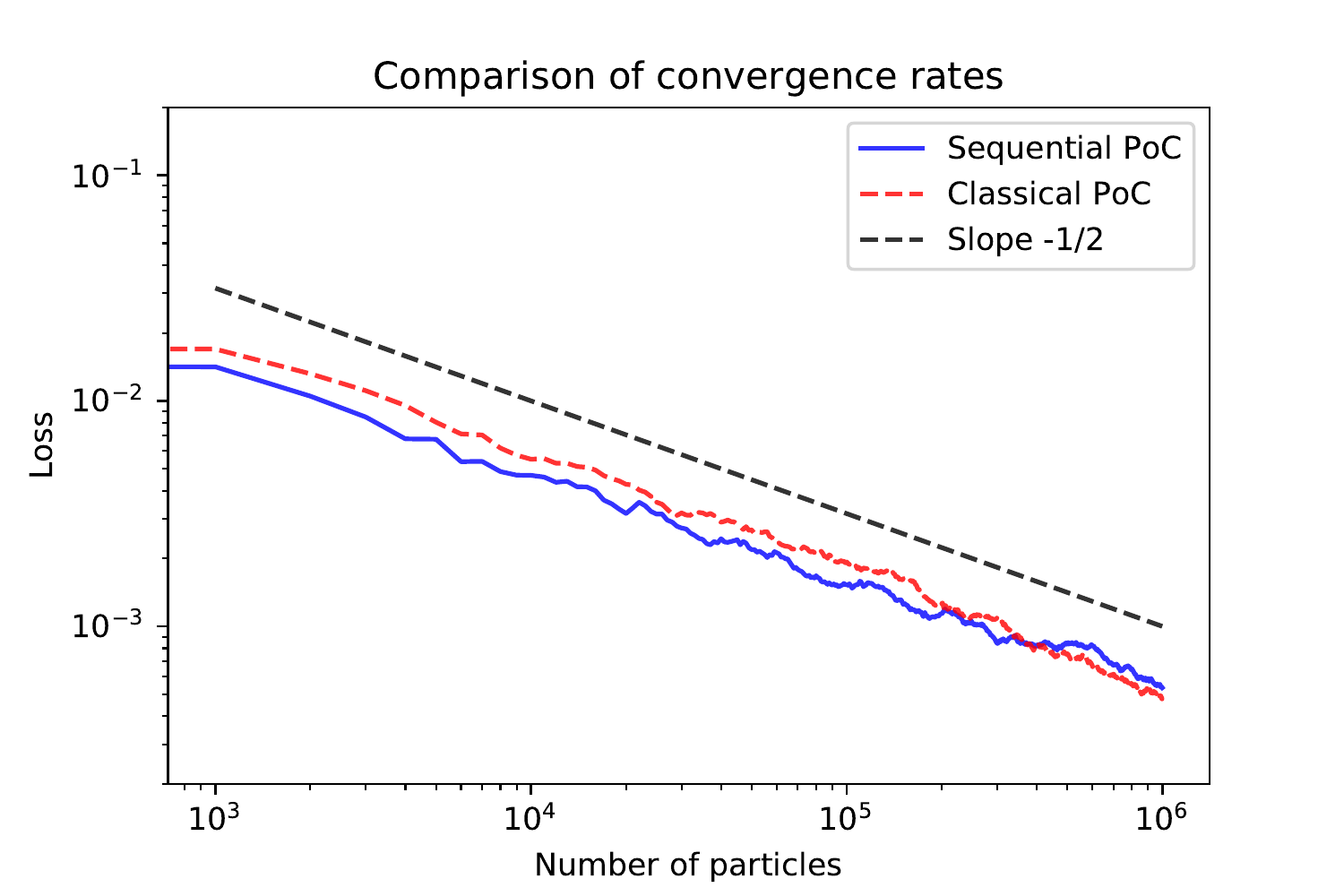}\includegraphics[width=0.5\textwidth]{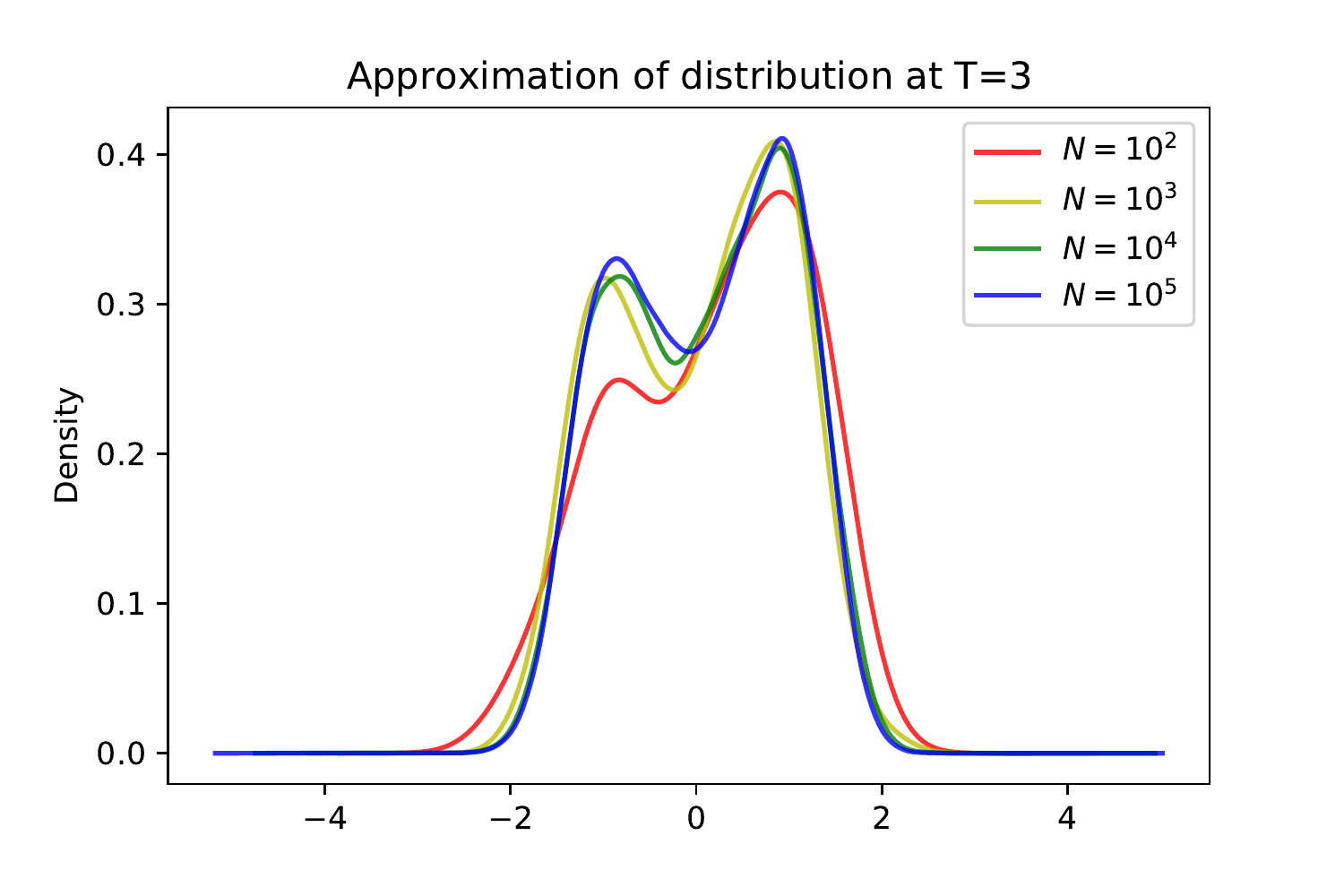}

\caption{\label{fig:Curie-Weiss-mean-field lattice model}Curie--Weiss mean-field
lattice model.}
\end{figure}
\end{example}
This paper is organized as follows. 
Section \ref{sec:2} presents the main results and some remarks. 
Auxiliary lemmas are proved in Section \ref{sec:3}. 
Sections \ref{sec:4} and \ref{sec:5} are devoted to the proofs of the main results, Theorems \ref{thm:APoC-W} and \ref{thm:time uniform-1}, respectively. 
More numerical examples are given in Section \ref{sec:6}.
In Appendix \ref{appendix}, we discuss the convergence in the Wasserstein distance for empirical measures on the path space.

\smallskip

We end the introduction with some notations. 
Denote by $|\cdot|$ and $\langle\cdot,\cdot\rangle$ the Euclidean norm and the inner product, respectively, in the Euclidean spaces, and by $\|\cdot\|$ the Frobenius norm of a matrix. 
Let $\mathscr{P}(E)$ be the space of all Borel probability measures on a normed space $(E,\|\cdot\|_{E})$. 
The $p$-Wasserstein  distance between $\mu,\nu\in\mathscr{P}(E)$ is defined as 
\[
\mathcal{W}_{p}(\mu,\nu):=\inf\big\{(\bE[\|\xi-\eta\|_{E}^{p}])^{1/p}:\cL(\xi)=\mu,\ \cL(\eta)=\nu\big\}.
\]
Moreover, for an increasing concave function $f:[0,\infty)\to[0,\infty)$ with $f(0)=0$, we define
\[
\mathcal{W}_{f}(\mu,\nu):=\inf\big\{\bE[f(\|\xi-\eta\|_{E})]:\cL(\xi)=\mu,\ \cL(\eta)=\nu\big\}.
\]
Denote by $\mathscr{P}_{p}$ the metric space of all probability measures
$\mu\in\mathscr{P}:=\mathscr{P}(\bR^{d})$ with $\|\mu\|_{p}:=[\int|x|^{p}\mu(\md x)]^{1/p}<\infty$,
equipped with the $p$-Wasserstein distance. 

\section{Main results\label{sec:2}}

Throughout this paper, we fix a \emph{decreasing} positive sequence \(\{\alpha_{n}\}_{n\geq 1}\) with \(\alpha_{1}=1\).
\(\alpha_{n}\) is the update rate of our weighted empirical measures.
We denote 
\[
\alpha_{\infty}:=\lim_{n\to\infty}\alpha_{n},\quad\underline{\alpha}:=\liminf_{n\to\infty}\frac{\alpha_{n}-\alpha_{n+1}}{\alpha_{n}^{2}},\quad\overline{\alpha}:=\limsup_{n\to\infty}\frac{\alpha_{n}-\alpha_{n+1}}{\alpha_{n}^{2}}.
\]
In the following, we will frequently encounter the weighted sum \(s_{n}\) recursively given by 
\begin{equation*}
    s_{n}=s_{n-1}+\alpha_{n}(x_{n}-s_{n-1}).
\end{equation*}
This gives 
\begin{equation*}
    s_{n}=\frac{\sum_{i=1}^{n}w_{i}x_{i}}{\sum_{i=1}^{n}w_{i}},
\end{equation*}
where \(w_{1}=1\) and \(w_{n}=\alpha_{n}\prod_{i=2}^{n}(1-\alpha_{i})^{-1}\) for \(n\geq 2\). 
For simplicity, we denote such a weighted sum \(s_{n}\) by \(\cK_{n}(x^{i})\).

We discuss the sequential PoC (SPoC) in both finite horizon and infinite horizon.
\subsection{Finite time SPoC}

We fix constants $T>0$ and \(p\geq 1\). 
\begin{assumption}[monotonicity]
\label{assu:lipschitz}There exists a constant $L\ge0$ such that 
\[
\begin{gathered}2\langle x-y,b(t,x,\mu)-b(t,y,\nu)\rangle+(2p-1)\|\sigma(t,x,\mu)-\sigma(t,y,\nu)\|^{2} \le L\big[|x-y|^{2}+\mathcal{W}_{2}(\mu,\nu)^{2}\big],\\
|b(t,0,\mu)|^{2}+\|\sigma(t,0,\mu)\|^{2}\le L\big[1+\|\mu\|_{2}^{2}\big]
\end{gathered}
\]
for all $t\in[0,T]$, $x,y\in\bR^{d}$ and $\mu,\nu\in\mathscr{P}_{2}$.
\end{assumption}
This monotonicity condition ensures the existence and uniqueness of
strong solutions of \eqref{eq:ADPS} and \eqref{eq:coupling-1} (cf.
\citet{wang2018distribution}). Extensive results on strong well-posedness
of McKean--Vlasov SDEs can be found in \citet{sznitman1991topics,buckdahn2017mean,bauer2018strong,mishura2020existence,hammersley2021mckean,du2021empirical}
and references therein. 

Our first main result is the sequential propagation of chaos in a finite horizon, i.e.,
the weighted empirical measure as well as the law of \(X_{t}^{n}\) converges to the law of the corresponding McKean--Vlasov process.
\begin{thm}

\label{thm:APoC-W} Let Assumption \ref{assu:lipschitz} be satisfied, $p\geq 1$, $\mu_{0}\in\mathscr{P}_{r}$ with $r>2p+(p-1)d$, and
\[
\gamma:=\frac{p}{2+(1-1/p)d}.
\]
Let $X^{n}$  be the solutions of \eqref{eq:ADPS} and \(\mu_{t}=\cL(X_t)\) given by~\eqref{eq:McKean--Vlasov} with the initial distribution~\(\mu_{0}\).
Then 
\begin{enumerate}
    \item if \(\overline{\alpha}<\gamma^{-1}\wedge(2-\alpha_{\infty})\), we have \[\sup_{0\leq t\leq T}\E[\cW_{2}(\mu_{t}^{n},\mu_{t})^{2p}+\cW_{2}(\cL(X^{n}_{t}),\mu_{t})^{2p}]\leq C\me^{CT}\alpha_{n}^{\gamma};\]
    \item if \(\gamma^{-1}\leq \overline{\alpha}<2-\alpha_{\infty}\), we have, for any \(\delta<1\wedge\underline{\alpha}\gamma\), \[\sup_{0\leq t\leq T}\E[\cW_{2}(\mu_{t}^{n},\mu_{t})^{2p}+\cW_{2}(\cL(X^{n}_{t}),\mu_{t})^{2p}]\leq C\me^{CT}\prod_{i=1}^{n}(1-\delta\alpha_{i});\]
    \item if \(\overline{\alpha}\geq 2\), we have, for any \(\delta<1\wedge\underline{\alpha}\gamma \wedge 2\gamma\),
    \[\sup_{0\leq t\leq T}\E[\cW_{2}(\mu_{t}^{n},\mu_{t})^{2p}+\cW_{2}(\cL(X^{n}_{t}),\mu_{t})^{2p}]\leq C\me^{CT}\prod_{i=1}^{n}(1-\delta\alpha_{i}).\]
\end{enumerate}
Here, the constant $C>0$ depends only on $d$, $p$, $L$, $\|\mu_{0}\|_{r}$. 
\end{thm}
\begin{rem}
    The case \(2-\alpha_{\infty}\leq \overline{\alpha}<2\) never appears because \(\overline{\alpha}=0\) as long as \(\alpha_{\infty}>0\).
\end{rem}
\begin{rem}
The estimate gives an elaborate characterization on how the stepsize
affects the asymptotic behavior of the system. The empirical measure
$\mu_{t}^{n}$ converges to $\mu_{t}$ as long as $\alpha_{\infty}=0$
and $\sum\alpha_{n}=\infty$; even if not so, we still give an upper
bound for the quantities concerned. The convergence may fail 
if $\alpha_{\infty}>0$ or $\sum\alpha_{n}<\infty$: intuitively,
$\alpha_{\infty}>0$ (or $\sum\alpha_{n}<\infty$) means that the
particles with large (or small) ordinals occupy too much weight in
$\mu_{t}^{n}$. The following example may demonstrate this point:
define a sequence of measures $\mu_{n}$ recursively as $\mu_{n}=\mu_{n-1}+\alpha_{n}(\delta_{w_{n}}-\mu_{n-1})$
and $\mu_{0}=\delta_{0}$, where $w_{n}\sim\mathcal{N}(0,1)$ are
i.i.d. The expected limit of $\mu_{n}$ is the standard Gaussian measure
(e.g., taking $\alpha=1/n$); however, for $\xi_{n}=\int x\,\mu_{n}(\md x)$,
by a simple computation one can see that $\E\xi_{n}=0$ but $\liminf_{n\to0}\E|\xi_{n}|^{2}>0$
whenever $\alpha_{\infty}>0$ or $\sum\alpha_{n}<\infty$; in other
words, $\xi_{n}$ never converges (in $L^{2}$) to zero, thus $\mu_{n}$
does not converge to the standard Gaussian measure.
\end{rem}
\begin{rem}
A typical choice of stepsize is $\alpha_{n}\sim n^{-r}$ with some
$r\in(0,1]$. In the case $r<1$, one has that
\[
\E[\mathcal{W}_{2}(\mu_{t}^{n},\mu_{t})^{2p}]= \mathcal{O}(n^{-2rp/(d+4)}).
\]
If $2p>(d+4)/r$, it follows from the Borel--Cantelli lemma that
$\mathcal{W}_{2}(\mu_{t}^{n},\mu_{t})$ converges to zero almost surely.
In contrast, in the case $\alpha_{n}=1/n$ (associated with the classical
empirical measures), we only have
\[
\bE[\mathcal{W}_{2}(\mu_{t}^{n},\mu_{t})^{2p}]=\mathcal{O}(n^{-\eps})\quad\text{for some\ }\eps<1
\]
even $p$ is large.
So, we cannot conclude the almost sure convergence from estimating a higher moment.
\end{rem}
\begin{rem}\label{rem:functional}
Notice that $\mu^{n}=(\mu_{t}^{n})_{t\le T}$ and $\mu=(\mu_{t})_{t\le T}$
are probability measures on the path space $C([0,T];\R^{d})$. 
By the sequential PoC and convergence of empirical measures on the path space, one can actually discuss the convergence of measures on the path space.
For instance, by taking $\alpha_n = 1/n$, we can show that
\[
\E[\cW_{2}(\mu^{n},\mu)^{2}]\le \mathcal{O}\Big(\frac{\ln\ln n}{\ln n}\Big),
\]
for the  Wasserstein distance  \(\cW_{2}\) on $\mathscr{P}(C([0,T];\R^{d}))$, see Corollary \ref{cor-path} for details.
\end{rem}
\begin{rem}
From the proof of Theorem \ref{thm:APoC-W}, one can see that if additionally
assuming
\[
\langle x-y,b(t,x,\mu)-b(t,y,\mu)\rangle\le-L_{0}|x-y|^{2},\quad\forall\,t\ge0,x,y\in\R^{d},\mu\in\mathscr{P}_{2}
\]
with a sufficiently large constant $L_{0}>0$, then the estimate
can be uniform in time. Of course, the additional condition would
be too restrictive in applications, so we discuss the uniform in time
PoC under some other conditions in the next subsection.
\end{rem}

\subsection{Uniform in time SPoC}

In this case, the diffusion term is required to be non-degenerate
and free of the law. To this end, we modify \eqref{eq:ADPS} to the
following form:
\begin{equation}
    \left\{
\begin{aligned}\md X_{t}^{n} & =b(t,X_{t}^{n},\mu_{t}^{n-1})\md t+\sigma(t,X_{t}^{n})\md W_{t}^{n}+\md B_{t}^{n},\\
\mu_{t}^{n} & =\mu_{t}^{n-1}+\alpha_{n}(\delta_{X_{t}^{n}}-\mu_{t}^{n-1}), \quad n\geq 2,
\end{aligned}\right.
\label{eq:ADPS-1}
\end{equation}
with \(X^{1}_{t}\equiv X^{1}_{0}\) and \(\mu_{t}^{1}=\delta_{X_{t}^{1}}\).
Here,  $\{W^{n}\}$ and $\{B^{n}\}$ are two independent sequences of multidimensional Brownian motions, the initial data $X_{0}^{n}\sim\mu_{0}$ are i.i.d. random variables independent of \(\{W^{n},B^{n}\}_{n\geq 1}\).
The corresponding McKean--Vlasov process is given by
\begin{equation}
    \left\{
\begin{aligned}\md X_{t} & =b(t,X_{t},\mu_{t})\md t+\sigma(t,X_{t})\md W_{t}+\md B_{t},\\
\mu_{t} & =\cL(X_{t}),
\end{aligned}
    \right.
\label{eq:McKean--Vlasov-1}
\end{equation}
where $W$ and $B$ are independent Brownian motions. 
\begin{assumption}[dissipativity]
\label{assu:uniform_H1}There is an increasing function $\kappa:(0,\infty)\to\mathbb{R}$
satisfying $\lim_{r\to\infty}\kappa(r)=:\kappa_{\infty}>0$ and $\lim_{r\to0+}r\cdot\kappa(r)=0$
such that
\[
\left\langle x-y,b(t,x,\mu)-b(t,y,\mu)\right\rangle +\frac{1}{2}\|\sigma(t,x)-\sigma(t,y)\|^{2}\le-\kappa(\|x-y\|)\|x-y\|^{2}
\]
 for all $t\in[0,\infty)$, $x,y\in\mathbb{R}^{d}$ and $\mu\in\mathscr{P}$. 
\end{assumption}
Define an increasing function $f\in C^{2}[0,\infty)$ with $f(0)=0$
and the derivative
\[
f^{\prime}(r)=\frac{1}{2}\int_{r}^{+\infty}s\cdot\exp\bigg(-\frac{1}{2}\int_{r}^{s}\tau\kappa(\tau)\md\tau\bigg)\md s.
\]
It is shown in Lemma \ref{lem:f_concave} that $f$ is concave. This
function has been used in \citet{Liu_2021} to study the mean-field
particle system.
\begin{assumption}[weak interaction]
\label{assu:uniform_H2}
There exists a constant \(0<\eta<1/f'(0)\) such that, for any $t\geq 0$, $x\in\bR^{d}$
and $\mu,\nu\in\mathscr{P}_{1}$,  it holds
\[
|b(t,x,\mu)-b(t,x,\nu)|\leq\eta\cW_{1}(\mu,\nu).
\]
\end{assumption}
The second main result is the following uniform in time PoC: the law
of $X_{t}^{n}$ as well as the (weighted) empirical measure $\mu_{t}^{n}$
converges uniformly in time to $\mu_{t}=\cL(X_{t})$.
\begin{thm}
\label{thm:time uniform-1}Suppose Assumptions \ref{assu:uniform_H1}
and \ref{assu:uniform_H2} hold.
Let $(X_{t}^{n},\mu_{t}^{n})$ and $(X_{t},\mu_{t})$ satisfy \eqref{eq:ADPS-1} and \eqref{eq:McKean--Vlasov-1}, respectively, with a common initial law $\mu_{0}$. 
In addition, we assume that
\begin{equation}
M_{p}:=\sup_{t\ge0}\E|X_{t}|^{p}<\infty\quad\text{with some}\quad p>d+2.\label{eq:Lp}
\end{equation}
Then,
\begin{enumerate}
    \item if \(\overline{\alpha}<2-\alpha_{\infty} \), we have \[\sup_{t\geq 0}\E[\cW_{f}(\mu_{t}^{n},\mu_{t})+\cW_{f}(\cL(X^{n}_{t}),\mu_{t})]\leq C\alpha_{n}^{\frac{1}{d+2}};\]
    \item if \(\overline{\alpha}\geq 2\), we have, for any \(\delta<1\wedge\frac{\underline{\alpha}}{d+2} \wedge \frac{2}{d+2}\),
    \[\sup_{t\geq 0}\E[\cW_{f}(\mu_{t}^{n},\mu_{t})+\cW_{f}(\cL(X^{n}_{t}),\mu_{t})]\leq C\prod_{i=1}^{n}(1-\delta\alpha_{i}).\]
\end{enumerate}
\end{thm}
\begin{rem}
The estimate is still valid with $\cW_{1}(\cdot,\cdot)$
in place of $\cW_{f}(\cdot,\cdot)$ and with a new constant $C$.
Indeed, Lemma \ref{lem:f_concave} shows $f(r)\ge r/\kappa_{\infty}$,
so $\cW_{1}(\cdot,\cdot)\le\kappa_{\infty}\cW_{f}(\cdot,\cdot)$ if
$\kappa_{\infty}<+\infty$. In the case $\kappa_{\infty}=+\infty$,
one can truncate the function $\kappa(\cdot)$, namely, by defining
$\kappa^{\lambda}(r):=\kappa(r)\wedge\lambda$.
This will change the function $f$ and increase $f'(0)$, but by taking $\lambda$
large enough, one can again ensure $f'(0)<1/\eta$ so that Assumption
\ref{assu:uniform_H2} is satisfied.
Finally, we get the desired
estimate by using $\kappa^{\lambda}(\cdot)$ instead of $\kappa(\cdot)$.
\end{rem}
\begin{rem}
The uniform $L^{p}$-boundedness \eqref{eq:Lp} is used to obtain
the convergence rate of empirical measures in the Wasserstein distance. Of
course, this condition can be ensured by imposing some growth condition
on the coefficients. For example, one can assume additionally that
there are constants $K_{0}<2(\kappa_{\infty}-\eta)/(p-2)$ and $K_{1}\ge0$
such that $\sup_{t\ge0}|b(t,0,\delta_{0})|\le K_{1}$ and
\[
\|\sigma(t,x)\|^{2}\le K_{0}|x|^{2}+K_{1},\quad\forall\,t\ge0,\,x\in\R^{d}.
\]
Indeed, this along with Assumptions \ref{assu:uniform_H1} and
\ref{assu:uniform_H2} implies that 
\[
\left\langle x,b(t,x,\delta_{0})\right\rangle +\frac{p-1}{2}\|\sigma(t,x)\|^{2}\le-\eta_{1}|x|^{2}+K_{2}
\]
with some $\eta_{1}>\eta$ and $K_{2}>0$, and \eqref{eq:Lp} can
be derived by It\^o's formula. Moreover, the index $p$ can
be reduced based on further improvement of Lemma \ref{lemma:iid}. 
\end{rem}
\begin{rem}
Under Assumptions \ref{assu:uniform_H1} and \ref{assu:uniform_H2},
one can actually obtain the exponential contractivity of the McKean--Vlasov
process \eqref{eq:McKean--Vlasov-1} by the reflection coupling method.
Specifically, for the solutions $\mu_{t},\nu_{t}$ of \eqref{eq:McKean--Vlasov-1}
with the initial distributions $\mu_{0},\nu_{0}$, respectively, one
can prove that
\[
\cW_{f}(\mu_{t},\nu_{t})\le\me^{-(f'(0)^{-1}-\eta)t}\cW_{f}(\mu_{0},\nu_{0}).
\]
See Remark \ref{rem:contrac} for further explanation. Similar results
can be found in \citet[Corollary 2.8]{Liu_2021} where the noise is
additive. For time homogeneous processes, this estimate implies the
existence and uniqueness of the invariant probability measure.
\end{rem}
\begin{example}
Consider the Curie--Weiss mean-field lattice model
\[
\md X_{t}=[-\beta(X_{t}^{3}-X_{t})+\beta K\bE X_{t}]\md t+\md W_{t}.
\]
This model is ferromagnetic or anti-ferromagnetic according to $K>0$
or $K<0$. By an elementary calculation, Assumption \ref{assu:uniform_H1}
is satisfied by taking $\kappa(r)=\beta(\frac{r^{2}}{4}-1)$, and
Assumption \ref{assu:uniform_H2} holds if
\begin{align*}
\frac{1}{K} & >\frac{\beta}{2}\int_{0}^{+\infty}r\cdot\exp\bigg(-\frac{\beta}{2}\int_{0}^{r}\frac{s^{3}-4s}{4}\md s\bigg)\md r\\
 & =\beta\me^{\beta/2}\int_{-1}^{+\infty}\me^{-\beta x^{2}/2}\md x=\sqrt{2\pi\beta\me^{\beta}}\Phi(\sqrt{\beta}),
\end{align*}
where $\Phi(x)$ is the distribution function of $\mathcal{N}(0,1)$.
\end{example}

\section{Auxiliary lemmas\label{sec:3}}

\subsection{Estimates for recursive inequalities\label{subsec:Estimates-for-recursive}}

\begin{lem}
    \label{lem-gw}
    Let \(\{A_{n}\}_{n\geq 0}\), \(\{B_{n}\}_{n\geq 0}\) be two decreasing positive sequences and \(\varepsilon>0\).
    We write 
    \begin{equation*}
        A_{n}=A_{0}\prod_{i=1}^{n}(1-\varepsilon\alpha_{i}),\quad  B_{n}=B_{0}\prod_{i=1}^{n}(1-\varepsilon\beta_{i}),
    \end{equation*}
    where \(\alpha_{i},\beta_{i}\in(0,\varepsilon^{-1})\).    
    Let \(\{s_{n}\}_{n\geq 0}\) be a nonnegative sequence satisfying
    \begin{equation*}
        s_{n}\leq (1-\varepsilon\alpha_{n})s_{n-1}+\alpha_{n}B_{n}.
    \end{equation*}
    Then, if \[\limsup_{n\to\infty}\frac{\beta_{n+1}}{\alpha_{n}}<1,\] there exists \(C>0\) such that
    \(
        s_{n}\leq C B_{n};
    \)
    and if \[\liminf_{n\to\infty}\frac{\beta_{n+1}}{\alpha_{n}}>1,\] there exists \(C>0\) such that
    \(
        s_{n}\leq C A_{n}.
    \)
\end{lem}
\begin{rem}
    The estimate will fail when \(\lim_{n\to\infty}\frac{\beta_{n+1}}{\alpha_{n}}=1\).
    For example, we take \(\alpha_{n}=\beta_{n}=n^{-1}\), then by direct calculation \(s_{n}\sim n^{-1}\ln n \).
    Thus, neither \(s_{n}\leq CA_{n}\) nor \(s_{n}\leq CB_{n}\) holds.
\end{rem}
\begin{proof}
    Without loss of generality, we assume   \begin{equation*}
        s_{n}= (1-\varepsilon\alpha_{n})s_{n-1}+\alpha_{n}B_{n}.
    \end{equation*}
    Let \(t_{n}=s_{n}A_{n}^{-1}\), and we get 
    \begin{equation*}
        t_{n}=t_{n-1}+ \alpha_{n}A_{n}^{-1}B_{n}.
    \end{equation*}
    Notice that for any \(1\leq i\leq n\) we have
    \begin{equation*}
        \sum_{j=1}^{i}\alpha_{j}A_{j}^{-1}=\frac{1}{\varepsilon}(A_{i}^{-1}-A_{0}^{-1})\quad\text{and}\quad \sum_{j=1}^{i-1}\beta_{j+1}B_{j}=\frac{1}{\varepsilon}(B_{1}-B_{i}).
    \end{equation*}
    This implies 
    \begin{equation}
        \label{eq:abel}
    \begin{aligned}
        t_{n}&=t_{0}+\sum_{i=1}^{n}\alpha_{i}A_{i}^{-1}B_{i}\\
        &=t_{0}+\alpha_{1}A_{1}^{-1}B_{1}+\sum_{i=2}^{n}\Bigl(\sum_{j=1}^{i}\alpha_{j}A_{j}^{-1}-\sum_{j=1}^{i-1}\alpha_{j}A_{j}^{-1}\Bigr) B_{i}\\
        &=t_{0}+\sum_{i=1}^{n-1}\Bigl(\sum_{j=1}^{i}\alpha_{j}A_{j}^{-1}\Bigr)(B_{i}-B_{i+1})+\sum_{j=1}^{n}\alpha_{j}A_{j}^{-1}B_{n}\\
        &=t_{0}+\sum_{i=1}^{n-1}\beta_{i+1}(A_{i}^{-1}-A_{0}^{-1})B_{i} + \frac{1}{\varepsilon}(A_{n}^{-1}-A_{0}^{-1})B_{n}\\
        &=t_{0}+\sum_{i=1}^{n-1}\beta_{i+1}A_{i}^{-1}B_{i}+\frac{1}{\varepsilon}(A_{n}^{-1}B_{n}-A_{0}^{-1}B_{1}).
    \end{aligned}
\end{equation}

    In the case of \(\limsup_{n\to\infty}\frac{\beta_{n+1}}{\alpha_{n}}<1\), there exist \(N,\delta>0\) such that \(\beta_{i+1}<(1-\delta)\alpha_{i}\) for any \(i\geq N\).
    By \eqref{eq:abel}, for any \(n\) we have
    \begin{align*}
        \sum_{i=1}^{n}\alpha_{i}A_{i}^{-1}B_{i} \leq (1-\delta)\sum_{i=1}^{n-1}\alpha_{i}A_{i}^{-1}B_{i}+\frac{1}{\varepsilon}A_{n}^{-1}B_{n}+ \sum_{i=1}^{N}\beta_{i+1}A_{i}^{-1}B_{i},
    \end{align*}
    and hence 
    \begin{equation*}
       t_{0}+ \sum_{i=1}^{n}\alpha_{i}A_{i}^{-1}B_{i}\leq C(1+A_{n}^{-1}B_{n})\leq C A_{n}^{-1}B_{n}.
    \end{equation*}
    This implies 
    \begin{align*}
        s_{n}=A_{n}t_{n}=A_{n}\Bigl(t_{0}+\sum_{i=1}^{n}\alpha_{i}A_{i}^{-1}B_{i}\Bigr)\leq CB_{n}.
    \end{align*}
    
    In the case of \(\liminf_{n\to\infty}\frac{\beta_{n+1}}{\alpha_{n}}>1\), there exist \(N,\delta>0\) such that \(\beta_{i+1}>(1+\delta)\alpha_{i}\) for any \(i\geq N\).
    By \eqref{eq:abel}, for any \(n\) we have 
    \begin{equation*}
        \sum_{i=1}^{n}\alpha_{i}A_{i}^{-1}B_{i}  \geq (1+\delta) \sum_{i=1}^{n-1}\alpha_{i}A_{i}^{-1}B_{i}+\alpha_{n}A_{n}^{-1}B_{n} -\frac{1}{\varepsilon}A_{0}^{-1}B_{1}-\sum_{i=1}^{N}\beta_{i+1}A_{i}^{-1}B_{i},
    \end{equation*} 
    and hence 
    \begin{equation*}
      \sum_{i=1}^{n-1}\alpha_{i}A_{i}^{-1}B_{i}\leq C.
    \end{equation*}
    Similarly, we get 
    \begin{equation*}
        s_{n}\leq CA_{n}.
    \end{equation*}
    The proof is complete.
\end{proof}

\begin{cor}
\label{cor:weight}
Let $\alpha_{n}$ be a decreasing positive sequence with \(\alpha_{1}=1\). 
We write \(\alpha_{\infty}=\lim_{n\to\infty}\alpha_{n}\), \(\overline{\alpha}=\limsup_{n\to\infty}\frac{\alpha_{n}-\alpha_{n+1}}{\alpha_{n}^{2}}\), and \(\underline{\alpha}=\liminf_{n\to\infty}\frac{\alpha_{n}-\alpha_{n+1}}{\alpha_{n}^{2}}\).
We set \(w_{1}=1\) and define \(w_{n}\) recursively by
\begin{equation*}
    \alpha_{n}=\frac{w_{n}}{\sum_{i=1}^{n}w_{i}}.
\end{equation*}
If 
\( \overline{\alpha}< 2-\alpha_{\infty},\)
then there exists \(C>0\) such that
\[
\theta_{n}:=\frac{\sum_{i=1}^{n}w_{i}^{2}}{(\sum_{i=1}^{n}w_{i})^{2}}\le C\alpha_{n};
\]
if \(\overline{\alpha}\geq 2\),
then for any \(0<\varepsilon <2\wedge \underline{\alpha}\) there exists \(C>0\) such that
\[
\theta_{n}\le C \prod_{i=1}^{n}(1-\varepsilon \alpha_{n}\mathbbm{1}_{\{\varepsilon \alpha_{n}<1\}}).
\]

\end{cor}

\begin{proof}
We have
\begin{align*}
\theta_{n}  =[1-(2-\alpha_{n})\alpha_{n}]\theta_{n}+\alpha_{n}^{2}.
\end{align*}
First, we assume \(\overline{\alpha}<2-\alpha_{\infty}\).
Let \(\varepsilon\in (\overline{\alpha},2-\alpha_{\infty})\) and assume  \(\varepsilon \alpha_{n}<1\) and   \(\varepsilon<2-\alpha_{n}\) for any \(n\geq N\).
We define 
\begin{equation}
    \label{eq:s}
    s_{n}=(1-\varepsilon \alpha_{n})s_{n}+\alpha_{n}^{2},\quad s_{N}=\theta_{N}.
\end{equation}
Then, for any \(n\geq N\)   we have  \(\theta_{n}\leq s_{n}\).
Notice that by setting \(B_{n}=\alpha_{n}\) we have \(\beta_{n}=\frac{\alpha_{n-1}-\alpha_{n}}{\varepsilon\alpha_{n-1}}\), and we further derive 
\begin{equation*}
    \limsup_{n\to\infty}\frac{\beta_{n+1}}{\alpha_{n}}=\overline{\alpha}<\varepsilon.
\end{equation*}
Therefore, by Lemma \ref{lem-gw} we obtain for any \(n\geq N\)
\begin{equation*}
    \theta_{n}\leq s_{n}\leq C B_{n}=C \alpha_{n}.
\end{equation*}
For the case \(\overline{\alpha}\geq 2\), let \(0<\varepsilon <2\wedge \underline{\alpha}\).
Define \(s_{n}\) as in \eqref{eq:s}.
Then similarly we have 
\begin{equation*}
    \liminf_{n\to\infty}\frac{\beta_{n+1}}{\alpha_{n}}=\underline{\alpha}>\varepsilon.
\end{equation*}
Therefore, by Lemma \ref{lem-gw} we obtain for any \(n\geq N\)
    \begin{equation*}
        \theta_{n}\leq s_{n}\leq C \prod_{i=N}^{n}(1-\varepsilon \alpha_{n}).
    \end{equation*}
    The proof is complete.
\end{proof}

\subsection{Wasserstein convergence of weighted empirical measures}

The convergence of (unweighted) empirical measures in the Wasserstein
distance has been extensively investigated in the literature (see
\citet{horowitz1994mean,fournier2015rate,ambrosio2019pde} and references
therein). Here we prove an estimate for weighted empirical measures,
based on the following density coupling lemma that can be proved analogously
as \citet[Lemma 2.2]{horowitz1994mean}.
\begin{lem}
\label{lem:coupling}Let $f$ and $g$ be probability density functions
on $\R^{d}$ such that
\[
\int_{\R^{d}}|x|^{r}[f(x)+g(x)]\md x<\infty,\quad r\ge1,
\]
and define $\mu(\md x)=f(x)\md x$ and $\nu(\md x)=g(x)\md x$. Then
one has
\[
\cW_{r}(\mu,\nu)^{r}\le C_{r}\int_{\R^{d}}|x|^{r}|f(x)-g(x)|\md x.
\]
\end{lem}
Then we have the following convergence result.

\begin{lem}
\label{lemma:iid}Let $r\in[1,\infty)$ and $p\in[1,\infty)$ and
$\xi_{n}$ be i.i.d. $\R^{d}$-valued random variables and $m:=\cL(\xi_{1})\in\mathscr{P}_{q}$
with $q>rp+(p-1)d$, and let
\[
m_{n}:=\frac{\sum_{i=1}^{n}w_{i}\delta_{\xi_{i}}}{\sum_{i=1}^{n}w_{i}}\quad\text{with }\ w_{n}>0,\ n=1,2,\dots.
\]
Then there is a constant $C$ depending only on $r,p,d$ and $\|\xi_{1}\|_{L^{q}}$,
such that 
\begin{align*}
\bE[\cW_{r}(m_{n},m)^{rp}] & \le C\theta_{n}^{\gamma_{r,p,d}}
\end{align*}
with
\[
\theta_{n}=\frac{\sum_{i=1}^{n}w_{i}^{2}}{(\sum_{i=1}^{n}w_{i})^{2}},\quad\gamma_{r,p,d}=\frac{rp}{2r+2(1-1/p)d}.
\]
\end{lem}
\begin{rem}
\label{rem:iid}Define $\alpha_{n}$ as the same way in Corollary \ref{cor:weight}.
 Then we have if \(\overline{\alpha}<2-\alpha_{\infty}\) 
\begin{equation*}
    \bE[\cW_{r}(m_{n},m)^{rp}] \le C\alpha_{n}^{\gamma_{r,p,d}};
\end{equation*}
if \(\overline{\alpha}\geq 2\), then for any \(0<\varepsilon<2\wedge \underline{\alpha}\),
\begin{equation*}
    \bE[\cW_{r}(m_{n},m)^{rp}] \le C\prod_{i=1}^{n}(1-\varepsilon\alpha_{i}\mathbbm{1}_{\{\varepsilon \alpha_{i}<1\}})^{\gamma_{r,p,d}}.
\end{equation*}
\end{rem}
\begin{proof}
Let $\phi_{\sigma}$ be the p.d.f. of the normal distribution $\Phi_{\sigma}=N(0,\sigma^{2}I_{d})$, and define mollified measures
\[
m_{n}^{\sigma}:=\Phi_{\sigma}*m_{n},\quad m^{\sigma}:=\Phi_{\sigma}*m.
\]
Denote the densities of  $m^{\sigma}$ and \(\mu^{\sigma}_{n}\) by \(\pi^{\sigma}\) and \(\pi^{\sigma}_{n}\) respectively. 
Then
$\pi_{n}^{\sigma}(\cdot)$ is given by 
\[
\pi_{n}^{\sigma}(x)=\frac{\sum_{i=1}^{n}w_{i}\phi_{\sigma}(x-\xi^{i})}{\sum_{i=1}^{n}w_{i}}=:\mathcal{K}_{n}(\phi_{\sigma}(x-\xi^{i})).
\]
Since \(q>rp+(p-1)d\), we notice 
\begin{equation*}
    (1+|x|)^{r-\frac{q}{p}}\in L^{p^{*}}(\bR^{d}),
\end{equation*}
where \(1/p+1/p^{*}=1\).
From Lemma \ref{lem:coupling} and H\"older's inequality, it follows that
\begin{equation}
    \label{eq:wr}
\begin{aligned}
\cW_{r}(m_{n}^{\sigma},m^{\sigma})^{rp} & \le\bigg(C_{r}\int|x|^{r}|\pi_{n}^{\sigma}(x)-\pi^{\sigma}(x)|\md x\bigg)^{p}\\
    &\leq C_{r}\|(1+|x|)^{r-q/p}\|_{L^{p^{*}}}^{p}\|(1+|x|)^{q/p}|\pi_{n}^{\sigma}(x)-\pi^{\sigma}(x)|\|_{L^{p}}^{p}\\
 & \le C_{r,p,q,d}\int(1+|x|)^{q}|\pi_{n}^{\sigma}(x)-\pi^{\sigma}(x)|^{p}\md x.
\end{aligned}
\end{equation}
For any i.i.d. random variables $\eta^{i}$ with zero mean, by Burkholder--Davis--Gundy inequality we derive
\begin{align*}
\E|\mathcal{K}_{n}(\eta_{i})|^{p} & =\frac{1}{(\sum_{i=1}^{n} w_{i})^{p}}\E\Big|\sum_{i=1}^{n} w_{i}\eta_{i}\Big|^{p}\le\frac{C_{p}}{(\sum_{i=1}^{n} w_{i})^{p}}\E\Big|\sum_{i=1}^{n} w_{i}^{2}|\eta_{i}|^{2}\Big|^{\frac{p}{2}}\\
 & \le C_{p}\frac{(\sum_{i=1}^{n} w_{i}^{2})^{\frac{p}{2}}}{(\sum_{i=1}^{n} w_{i})^{p}}\E\Big|\Bigl(\sum_{i=1}^{n} w_{i}^{2}\Bigr)^{-1}\sum_{i=1}^{n} w_{i}^{2}|\eta_{i}|^{2}\Big|^{\frac{p}{2}}\\
 & \le C_{p}\frac{(\sum_{i=1}^{n} w_{i}^{2})^{\frac{p}{2}}}{(\sum_{i=1}^{n}w_{i})^{p}}\E\Big[\Bigl(\sum_{i=1}^{n} w_{i}^{2}\Bigr)^{-1}\sum_{i=1}^{n} w_{i}^{2}|\eta_{i}|^{p}\Big]\\
 & =C_{p}\err^{\frac{p}{2}}\E|\eta_{1}|^{p}.
\end{align*}
Thus, we have 
\begin{align*}
\bE|\pi_{n}^{\sigma}(x)-\pi_{}^{\sigma}(x)|^{p} & =\bE|\mathcal{K}_{n}(\phi_{\sigma}(x-\xi^{i})-\bE\phi_{\sigma}(x-\xi))|^{p}\\
 & \le C_{p}\err^{\frac{p}{2}}\bE|\phi_{\sigma}(x-\xi)|^{p}.
\end{align*}
We observe that $\phi_{\sigma}^{p}(x)=(2\pi)^{-\frac{pd}{2}}\sigma^{-pd}e^{-\frac{p|x|^{2}}{2\sigma^{2}}}=p^{-\frac{d}{2}}(2\pi)^{\frac{d}{2}(1-p)}\sigma^{d(1-p)}\phi_{\sigma/\sqrt{p}}(x)$, and we calculate
\begin{align*}
\bE|\phi_{\sigma}(x-\xi)|^{p } & =\int\phi_{\sigma}^{p}(x-y)m(\md y)\\
 & =p^{-\frac{d}{2}}(2\pi)^{\frac{d}{2}(1-p)}\sigma^{d(1-p)}\int\phi_{\sigma/\sqrt{p}}(x-y)m(\md y)\\
 & =C_{p,d}\sigma^{d(1-p)}\int\phi_{\sigma/\sqrt{p}}(x-y)m(\md y).
\end{align*}
Plug above estimates into \eqref{eq:wr}, then we get
\begin{align*}
\bE[\cW_{r}(m_{n}^{\sigma},m^{\sigma})^{rp}] & \le C_{r,p,q,d}\int(1+|x|)^{q}\E|\pi_{n}^{\sigma}(x)-\pi^{\sigma}(x)|^{p}\md x\\
 & \le C_{r,p,q,d}\sigma^{d(1-p)}\err^{\frac{p}{2}}\iint(1+|x|)^{q}\phi_{\sigma/\sqrt{p}}(x-y)\,m(\md y)\md x\\
 &  \le C_{r,p,q,d}\sigma^{d(1-p)}\err^{\frac{p}{2}}\iint(1+|x|+|y|)^{q}\phi_{\sigma/\sqrt{p}}(x)\md x \,m(\md y)\\
 & \le C_{r,p,q,d}\sigma^{d(1-p)}\err^{\frac{p}{2}}\E(1+|\xi_{1}|)^{q}.
\end{align*}
On the other hand, it is easily seen that 
\[
\cW_{r}(m_{n}^{\sigma},m_{n})^{rp}+\cW_{r}(m^{\sigma},m)^{rp}\le C\sigma^{rp}.
\]
Noticing
\begin{equation*}
    \cW_{r}(m_{n},m)\leq \cW_{r}(m_{n},m_{n}^{\sigma})+\cW_{r}(m_{n}^{\sigma},m^{\sigma})+\cW_{r}(m^{\sigma},m),
\end{equation*}
we derive 
\begin{equation*}
    \E[\cW_{r}(m_{n},m)^{rp}]\leq C (\sigma^{rp}+\sigma^{d(1-p)}\theta_{n}^{\frac{p}{2}}).
\end{equation*}
We conclude the proof by taking $\sigma=\theta_{n}^{1/[2r+2(1-p^{-1})d]}$.
\end{proof}

\section{Proof of Theorem \ref{thm:APoC-W}\label{sec:4}}
Recall the sequential propagation system
\begin{equation}
\left\{
\begin{aligned}\md X_{t}^{n} & =b(t,X_{t}^{n},\mu_{t}^{n-1})\md t+\sigma(t,X_{t}^{n},\mu_{t}^{n-1})\md W_{t}^{n};\\
\mu_{t}^{n} & =\mu_{t}^{n-1}+\alpha_{n}(\delta_{X_{t}^{n}}-\mu_{t}^{n-1}),
\end{aligned}
\right.
\end{equation}
with \(X^{1}_{t}\equiv X^{1}_{0}\) and \(\mu_{t}^{1}=\delta_{X_{t}^{1}}\).
To apply the synchronous coupling method, we introduce a sequence of i.i.d. McKean--Vlasov processes $Y^{n}$
defined by 
\begin{equation}
    \left\{
\begin{aligned}\md Y_{t}^{n} & =b(t,Y_{t}^{n},\mu_{t})\md t+\sigma(t,Y_{t}^{n},\mu_{t})\md W_{t}^{n},\\
Y_{0}^{n} & =X_{0}^{n}\sim\mu_{0},
\end{aligned}\right.
\label{eq:coupling-1}
\end{equation}
where $\mu_{t}=\cL(X_{t})$ from (\ref{eq:McKean--Vlasov}) is also the law of $Y_{t}^{n}$.

Fix a constant $\lambda\ge0$ that will be specified later. 
Define the discounted difference
\[
\Delta_{t}^{n,\lambda}=\me^{-\lambda t}(X_{t}^{n}-Y_{t}^{n}).
\]
Applying It\^o's formula we compute
\begin{align*}
 & \quad\md|\Delta_{t}^{n,\lambda}|^{2p}\\
 & =-2p\lambda|\Delta_{t}^{n,\lambda}|^{2p}\md t+p\me^{-2\lambda t}|\Delta_{t}^{n,\lambda}|^{2p-2}\Big\{2(X_{t}^{n}-Y_{t}^{n})^{\tr}[b(t,X_{t}^{n},\mu_{t}^{n-1})-b(t,Y_{t}^{n},\text{\ensuremath{\mu}}_{t})]\\
 & \quad+2(p-1)|X_{t}^{n}-Y_{t}^{n}|^{-2}\big|(X_{t}^{n}-Y_{t}^{n})^{\tr}[\sigma(t,X_{t}^{n},\mu_{t}^{n-1})-\sigma(t,Y_{t}^{n},\text{\ensuremath{\mu}}_{t})]\big|^{2}\\
 &\quad+\|\sigma(t,X_{t}^{n},\mu_{t}^{n-1})-\sigma(t,Y_{t}^{n},\text{\ensuremath{\mu}}_{t})\|^{2}\Big\}\md t\\
 & \quad+2p\me^{-\lambda t}|\Delta_{t}^{n,\lambda}|^{2p-2}(\Delta_{t}^{n,\lambda})^{\tr}[\sigma(t,X_{t}^{n},\mu_{t}^{n-1})-\sigma(t,Y_{t}^{n},\text{\ensuremath{\mu}}_{t})]\md W_{t}^{n}.
\end{align*}
By Assumption \ref{assu:lipschitz} and Young's inequality we have
\begin{equation}
\begin{aligned}\md|\Delta_{t}^{n,\lambda}|^{2p} & \le\big[(C-2p\lambda)|\Delta_{t}^{n,\lambda}|^{2p}+C|\Delta_{t}^{n,\lambda}|^{2p-2}\me^{-2\lambda t}\mathcal{W}_{2}(\mu_{t}^{n},\mu_{t})^{2}\big]\md t\\
 & \quad+2p\me^{-\lambda t}|\Delta_{t}^{n,\lambda}|^{2p-2}(\Delta_{t}^{n,\lambda})^{\tr}[\sigma(t,X_{t}^{n},\mu_{t}^{n-1})-\sigma(t,Y_{t}^{n},\text{\ensuremath{\mu}}_{t})]\md W_{t}^{n}\\
 & \le\big[(C-2p\lambda)|\Delta_{t}^{n,\lambda}|^{2p}+C\me^{-2p\lambda t}\mathcal{W}_{2}(\mu_{t}^{n-1},\mu_{t})^{2p}\big]\md t\\
 & \quad+2p\me^{-\lambda t}|\Delta_{t}^{n,\lambda}|^{2p-2}(\Delta_{t}^{n,\lambda})^{\tr}[\sigma(t,X_{t}^{n},\mu_{t}^{n-1})-\sigma(t,Y_{t}^{n},\text{\ensuremath{\mu}}_{t})]\md W_{t}^{n},
\end{aligned}
\label{eq:3-01}
\end{equation}
where $C=C(p,L)$. 
Taking the expectation and by Gronwall inequality, we have 
\begin{equation}
    \label{eq:delta}
    \E|\Delta_{t}^{n,\lambda}|^{2p}\leq \me^{(C-2p\lambda)t}C\int_{0}^{t}\me^{-2p\lambda s}\E[\cW_{2}(\mu_{s}^{n-1},\mu_{s})^{2p}]\md s.
\end{equation}
We take \(\varepsilon\in(0,1)\), \(\lambda=\frac{(2-\varepsilon)C}{(1-\varepsilon)2p}\) and further derive 
\begin{equation}
    \label{eq:gw}
    \begin{aligned}
    \int_{0}^{T}\E|\Delta_{t}^{n,\lambda}|^{2p}\md t &\leq \frac{C}{2p\lambda-C}\int_{0}^{T}\me^{-2p\lambda t}\E[\cW_{2}(\mu_{t}^{n-1},\mu_{t})^{2p}]\md t\\
    &=(1-\varepsilon)\int_{0}^{T}\me^{-2p\lambda t}\E[\cW_{2}(\mu_{t}^{n-1},\mu_{t})^{2p}]\md t.
    \end{aligned}
\end{equation}
It follows from the property of the Wasserstein distance and Lemma \ref{lemma:iid} that
\begin{align*}
\E[\mathcal{W}_{2}(\mu_{t}^{n-1},\mu_{t})^{2p}] & \le (1+\varepsilon) \bE[\mathcal{W}_{2}(\mathcal{K}_{n-1}(\delta_{X_{t}^{i}}),\mathcal{K}_{n-1}(\delta_{Y_{t}^{i}}))^{2p}]+C_{p}\bE[\mathcal{W}_{2}(\mathcal{K}_{n-1}(\delta_{Y_{t}^{i}}),\mu_{t})^{2p}]\\
 & \leq (1+\varepsilon)\bE[\mathcal{K}_{n-1}(|X_{t}^{i}-Y_{t}^{i}|^{2p})]+C_{p}\theta_{n}^{\gamma}.
\end{align*}
By \(s_{n}\) we denote \(\int_{0}^{T}\cK_{n}(\E|\Delta_{t}^{i,\lambda}|^{2p})\md t\).
Combining the above inequality with \eqref{eq:gw}, we obtain 
\begin{align*}
    s_{n}&=s_{n-1}+\alpha_{n}\Bigl(\int_{0}^{T}\E|\Delta_{t}^{n,\lambda}|^{2p}\md t-s_{n-1}\Bigr)\\
    &\leq s_{n-1}+\alpha_{n}(1-\varepsilon)[(1+\varepsilon)s_{n-1}+C_{p}\theta_{n}^{\gamma}]-\alpha_{n}s_{n-1}\\
    &\leq (1-\varepsilon^{2}\alpha_{n})s_{n-1} +\alpha_{n}C\theta_{n}^{\gamma}.
\end{align*}
Here, \(\gamma=p/(2+(1-1/p)d)\).

In the case of \(\overline{\alpha}<2-\alpha_{\infty}\), we have \(\theta_{n}\leq C \alpha_{n}\).
Set \(B_{n}=\alpha_{n}^{\gamma}\), then we have \(\beta_{n}=\frac{\alpha_{n-1}^{\gamma}-\alpha_{n}^{\gamma}}{\varepsilon^{2}\alpha_{n-1}^{\gamma}}\).
Thus, if \begin{equation*}
    \limsup_{n\to\infty}\frac{\alpha_{n}^{\gamma}-\alpha_{n+1}^{\gamma}}{\alpha_{n}^{\gamma+1}}=\gamma\overline{\alpha}<1,
\end{equation*}
by Lemma \ref{lem-gw} we have \(s_{n}\leq C\alpha_{n}^{\gamma}\).
If\begin{equation*}
    \limsup_{n\to\infty}\frac{\alpha_{n}^{\gamma}-\alpha_{n+1}^{\gamma}}{\alpha_{n}^{\gamma+1}}=\gamma\overline{\alpha}\geq1,
\end{equation*} 
then by Lemma \ref{lem-gw} we have 
\begin{equation*}
    s_{n}\leq C\prod_{i=1}^{n}(1-\delta \alpha_{i})
\end{equation*}
for any \(\delta<1\wedge\underline{\alpha}\gamma\).

In the case of \(\overline{\alpha}\geq 2\), we have, for any \(\eta<2\wedge \underline{\alpha}\),
\begin{equation*}
    \theta_{n}\leq C \prod_{i=1}^{n}(1-\eta \alpha_{i}\mathbbm{1}_{\{\eta\alpha_{i}<1\}})\leq C\prod_{i=1}^{n-1}(1-\eta \alpha_{i}\mathbbm{1}_{\{\eta\alpha_{i}<1\}}).
\end{equation*} 
Set \(B_{n}=\prod_{i=1}^{n-1}(1-\eta \alpha_{i}\mathbbm{1}_{\{\eta\alpha_{i}<1\}})^{\gamma}\), then we derive 
\begin{equation*}
    \beta_{n}=\frac{1}{\varepsilon^{2}}[1-(1-\eta \alpha_{n-1}\mathbbm{1}_{\{\eta\alpha_{n-1}<1\}})^{\gamma}].
\end{equation*}
Thus, we compute 
\begin{equation*}
    \lim_{n\to\infty}\frac{\beta_{n+1}}{\alpha_{n}}=\lim_{n\to\infty} \frac{1-(1-\eta\alpha_{n})^{\gamma}}{\varepsilon^{2}\alpha_{n}}=\frac{\eta\gamma}{\varepsilon^{2}}.
\end{equation*}
By Lemma \ref{lem-gw}, we obtain, for any \(\delta<1\wedge\underline{\alpha}\gamma \wedge 2\gamma\),
\begin{equation*}
    s_{n}\leq C\prod_{i=1}^{n}(1-\delta\alpha_{i}).
\end{equation*}

To sum up, we plug above estimates into \eqref{eq:delta} and obtain:
\begin{enumerate}
    \item if \(\overline{\alpha}<\gamma^{-1}\wedge(2-\alpha_{\infty})\), we have \[\sup_{0\leq t\leq T}\E[\cW_{2}(\mu_{t}^{n},\mu_{t})^{2p}+\cW_{2}(\cL(X^{n}_{t}),\mu_{t})^{2p}]\leq C\me^{CT}\alpha_{n}^{\gamma};\]
    \item if \(\gamma^{-1}\leq \overline{\alpha}<2-\alpha_{\infty}\), we have, for any \(\delta<1\wedge\underline{\alpha}\gamma\), \[\sup_{0\leq t\leq T}\E[\cW_{2}(\mu_{t}^{n},\mu_{t})^{2p}+\cW_{2}(\cL(X^{n}_{t}),\mu_{t})^{2p}]\leq C\me^{CT}\prod_{i=1}^{n}(1-\delta\alpha_{i});\]
    \item if \(\overline{\alpha}\geq 2\), we have, for any \(\delta<1\wedge\underline{\alpha}\gamma \wedge 2\gamma\),
    \[\sup_{0\leq t\leq T}\E[\cW_{2}(\mu_{t}^{n},\mu_{t})^{2p}+\cW_{2}(\cL(X^{n}_{t}),\mu_{t})^{2p}]\leq C\me^{CT}\prod_{i=1}^{n}(1-\delta\alpha_{i}).\]
\end{enumerate}

\section{Proof of Theorem \ref{thm:time uniform-1}\label{sec:5}}

Recall the function $f\in C^{2}[0,\infty)$ with $f(0)=0$ and
\[
f^{\prime}(r)=\frac{1}{2}\int_{r}^{\infty}s\cdot\exp\bigg(-\frac{1}{2}\int_{r}^{s}\tau\kappa(\tau)\md\tau\bigg)\md s,
\]
which actually satisfies the ordinary differential equation (ODE):
\begin{equation}
2f^{\prime\prime}(r)-r\kappa(r)f^{\prime}(r)=-r.\label{eq:f''-rkf' =00003D -r}
\end{equation}

\begin{lem}
\label{lem:f_concave}$f$ is concave; consequently, $\kappa_{\infty}^{-1}\le f(r)/r\le f'(0)$
for all $r>0$.
\end{lem}
\begin{proof}
Since $\kappa_{\infty}>0$, there exists $R=\inf\{r>0:\kappa(s)>0,\forall s>r\}$.
Then for $r>R$,
\begin{align*}
f^{\prime}(r) & =\frac{1}{2}\int_{r}^{\infty}s\me^{-\frac{1}{2}\int_{r}^{s}\tau\kappa(\tau)\md\tau}\md s\\
 & \le\frac{1}{2\kappa(r)}\int_{r}^{\infty}s\kappa(s)\me^{-\frac{1}{2}\int_{r}^{s}\tau\kappa(\tau)\md\tau}\md s\\
 & =\frac{1}{\kappa(r)}\big(1-\me^{-\frac{1}{2}\int_{r}^{\infty}\tau\kappa(\tau)\md\tau}\big)\le\frac{1}{\kappa(r)},
\end{align*}
which along with \eqref{eq:f''-rkf' =00003D -r} implies $f''(r)\le0$;
for $r\le R$, we have $f^{\prime}(r)\kappa(r)\le0$. 
This implies
\begin{equation}
f^{\prime\prime}(r)=f^{\prime}(r)\kappa(r)r-r\le0.\label{eq:f}
\end{equation}
Hence, $f$ is concave. Consequently, $f'(r)\le f'(0)$. Moreover,
\begin{align*}
f^{\prime}(r) & \ge\frac{1}{2}\int_{r}^{+\infty}s\me^{-\frac{1}{2}\int_{r}^{s}\tau\kappa_{\infty}\md\tau}\md s=\frac{1}{\kappa_{\infty}}.
\end{align*}
This concludes the proof.
\end{proof}

Now we turn to the proof of Theorem \ref{thm:time uniform-1}.
Recall that for \(n\geq 2\) we recursively define
\begin{equation*}
    \left\{
\begin{aligned}\md X_{t}^{n} & =b(t,X_{t}^{n},\mu_{t}^{n-1})\md t+\sigma(t,X_{t}^{n})\md W_{t}^{n}+\md B_{t}^{n},\\
\mu_{t}^{n} & =\mu_{t}^{n-1}+\alpha_{n}(\delta_{X_{t}^{n}}-\mu_{t}^{n-1}), 
\end{aligned}\right.
\end{equation*}
with \(X^{1}_{t}\equiv X^{1}_{0}\), \(\mu_{t}^{1}=\delta_{X_{t}^{1}}\), and its corresponding McKean--Vlasov process 
\begin{equation*}
    \left\{\begin{aligned}
        \md Y_{t}^{n}&=b(t,Y_{t}^{n},\mu_{t})\md t+\sigma(t,Y_{t}^{n})\md W_{t}^{n}+\md B_{t}^{n},\\
        \mu_{t}&=\cL(Y_{t}),
    \end{aligned}\right.
\end{equation*}
with \( Y_{0}^{n}=X_{0}^{n}\).
Following the idea of \cite{Reflectioncouplings,uniform},  we take smooth functions $\lambda^{n},\pi^{n}:\bR^{d}\to\bR$ satisfying
\begin{equation}
    \label{eqn-lp}
|\lambda^{n}(x)|^{2}+|\pi^{n}(x)|^{2}=1\text{ for all \ensuremath{x\in\bR^{d}}},\qquad\lambda^{n}(x)=\begin{cases}
1 & \text{if }|x|\geq\delta_{n},\\
0 & \text{if }|x|\leq\delta_{n}/2.
\end{cases}
\end{equation}
Here a positive sequence $\{\delta_{n}\}_{n=1}^{\infty}$ will be specified later.
Instead of the synchronous coupling \((X^{n},Y^{n})\), we define a coupling of reflection \((\widehat{X}^{n},\widehat{Y}^{n})\)  as
\[
    \left\{
\begin{aligned}\md \widehat{X}_{t}^{n} & =b(t,\widehat{X}_{t}^{n},\widehat{\mu}_{t}^{n-1})\md t+\sigma(t,\widehat{X}_{t}^{n})\md W_{t}^{n} +\lambda^{n}(\Delta_{t}^{n})\md B_{t}^{n}+\pi^{n}(\Delta_{t}^{n})\md\widehat{B}_{t}^{n},\\
\widehat{\mu}_{t}^{n} & =\widehat{\mu}_{t}^{n-1}+\alpha_{n}(\delta_{\widehat{X}_{t}^{n}}-\widehat{\mu}_{t}^{n-1});\\
\end{aligned}\right.
\]
and 
\[\left\{
\begin{aligned}
    \md \widehat{Y}_{t}^{n} & =b(t,\widehat{Y}_{t}^{n},\mu_{t})\md t+\sigma(t,\widehat{Y}_{t}^{n})\md W_{t}^{n} +\lambda^{n}(\Delta_{t}^{n})\big(\id-2e_{t}^{n}(e_{t}^{n})^{\tr}\big)\md B_{t}^{n}+\pi^{n}(\Delta_{t}^{n})\md\widehat{B}_{t}^{n},\\
    \mu_{t}&=\cL(Y_{t}),
\end{aligned}    \right.
\]
where   $\{\widehat{B}^{n}\}$ is an additional sequence of independent Brownian motions which is independent of \(\{B^{n},W^{n},X^{n}_{0}\}\), $\Delta_{t}^{n}=\widehat{X}_{t}^{n}-\widehat{Y}_{t}^{n}$, and 
\begin{equation*}
e_{t}^{n}=\begin{cases}
\Delta_{t}^{n}/|\Delta_{t}^{n}| & \text{if }\Delta_{t}^{n}\neq0,\\
0 & \text{if }\Delta_{t}^{n}=0.
\end{cases}
\end{equation*}
We set initial data \(\widehat{X}^{n}_{0}=\widehat{Y}^{n}_{0}=X^{n}_{0}=Y^{n}_{0}\) and \(\widehat{\mu}_{t}^{1}=\delta_{\widehat{X}^{1}_{t}}\).
By Levy's characterization and \eqref{eqn-lp}, we can show \((\widehat{X}^{1},\dots,\widehat{X}^{n})\) shares the same joint law with \((X^{1},\dots,X^{n})\) by induction.
Thus, \(\cL(\widehat{Y}^{n}_{t})=\cL(Y^{n}_{t})=\mu_{t}\) and \(\widehat{\mu}^{n}_{t}\) have the same law with \(\mu^{n}_{t}\) as random variables on \(\sP(\bR^{d})\).

We compute that
\begin{equation*}
\begin{aligned}\md\Delta_{t}^{n} & =\md(\widehat{X}_{s}^{n}-\widehat{Y}_{t}^{n})\\
 & =[b(t,\widehat{X}_{t}^{n},\widehat{\mu}_{t}^{n-1})-b(t,\widehat{Y}_{t}^{n},\mu_{t})]\md t\\
 &\quad+\bigl[\sigma(t,\widehat{X}_{t}^{n})-\sigma(t,\widehat{Y}_{t}^{n})\bigr]\md W_{t}^{n} +2\lambda^{n}(\Delta_{t}^{n})e_{t}^{n}(e_{t}^{n})^{\tr}\md B_{t}^{n}\\
 & =[b(t,\widehat{X}_{t}^{n},\widehat{\mu}_{t}^{n-1})-b(t,\widehat{Y}_{t}^{n},\widehat{\mu}_{t}^{n-1})]\md t+[b(t,\widehat{Y}_{t}^{n},\widehat{\mu}_{t}^{n-1})-b(t,\widehat{Y}_{t}^{n},\mu_{t})]\md t\\
 & \quad+[\sigma(t,\widehat{X}_{t}^{n})-\sigma(t,\widehat{Y}_{t}^{n})]\md W_{t}^{n}+2\lambda^{n}(\Delta_{t}^{n})e_{t}^{n}(e_{t}^{n})^{\tr}\md B_{t}^{n}\\
 & :=(\Delta b_{t}^{n,1}+\Delta b_{t}^{n,2})\md t+\Delta \sigma_{t}^{n}\md W_{t}^{n}+2\lambda^{n}(\Delta_{t}^{n})e_{t}^{n}(e_{t}^{n})^{\tr}\md B_{t}^{n}.
\end{aligned}
\label{eq:Delta_n}
\end{equation*}
By the It\^o--Tanaka formula, we have
\begin{equation}
    \label{eqn-ab}
\begin{aligned}
\md|\Delta_{t}^{n}| =\frac{1}{2|\Delta_{t}^{n}|}\big(2\langle \Delta_{t}^{n},\Delta b_{t}^{n,1} +\Delta b_{t}^{n.2}\rangle&+\|\Delta \sigma_{t}^{n}\|^{2}\big)\md t \\
&+\left\langle e_{t}^{n},\Delta \sigma_{t}^{n}\right\rangle \md W_{t}^{n}+2\lambda^{n}(\Delta_{t}^{n})(e_{t}^{n})^{\tr}\md B_{t}^{n}.
\end{aligned}
\end{equation}
By Assumption \ref{assu:uniform_H1}, we obtain 
\begin{align*}
\langle\Delta_{t}^{n},\Delta b_{t}^{n,1}\rangle+\frac{1}{2}\|\Delta\sigma_{t}^{n}\|^{2} & \le-\kappa(|\Delta_{t}^{n}|)|\Delta_{t}^{n}|^{2}.
\end{align*}
By Lemma \ref{lemma:iid} and Assumption \ref{assu:uniform_H2},  we can further get
\begin{align*}
    |\Delta b^{n,2}_{t}|&\leq \eta \cW_{1}(\widehat{\mu}^{n-1}_{t},\mu_{t})=\eta \cW_{1}(\cK_{n-1}(\delta_{\widehat{X}_{t}^{i}}),\mu_{t})\\
    &\leq\eta \cW_{1}(\cK_{n-1}(\delta_{\widehat{X}_{t}^{i}}),\cK_{n-1}(\delta_{\widehat{Y}_{t}^{i}}))+\eta\cW_{1}(\cK_{n-1}(\delta_{\widehat{Y}_{t}^{i}}),\mu_{t})\\
    &\leq \eta\cK_{n-1}(|\Delta_{t}^{i}|)+C \eta \theta_{n}^{\frac{1}{d+2}}.
\end{align*}
Plugging above estimates into \eqref{eqn-ab}, we have 
\begin{equation*}
    \begin{aligned}
    \md|\Delta_{t}^{n}|\leq -\kappa(|\Delta_{t}^{n}|)|\Delta^{n}_{t}|\md t &+ (\eta\cK_{n-1}(|\Delta_{t}^{i}|)+C \eta \theta_{n}^{\frac{1}{d+2}})\md t\\
    &+\left\langle e_{t}^{n},\Delta \sigma_{t}^{n}\right\rangle \md W_{t}^{n}+2\lambda^{n}(\Delta_{t}^{n})(e_{t}^{n})^{\tr}\md B_{t}^{n}.
    \end{aligned}
\end{equation*}
Since \(f''\leq 0\), by It\^o's formula we show
\begin{align*}
\md f(|\Delta_{t}^{n}|) & \le-\kappa(|\Delta_{t}^{n}|)|\Delta_{t}^{n}|f^{\prime}(|\Delta_{t}^{n}|)\md t+2f^{\prime\prime}(|\Delta_{t}^{n}|)\vert\lambda^{n}(\Delta_{t}^{n})\vert^{2}\md t\\
 & \quad+f'(|\Delta_{t}^{n}|)(\eta\mathcal{K}_{n-1}(|\Delta_{t}^{i}|)+C \eta \theta_{n}^{\frac{1}{d+2}})\md t\\
 & \quad+f'(|\Delta_{t}^{n}|)\left\langle e_{t}^{n},\Delta\sigma_{t}^{n}\right\rangle \md W_{t}^{n}+2f'(|\Delta_{t}^{n}|)\lambda^{n}(\Delta_{t}^{n})(e_{t}^{n})^{\tr}\md B_{t}^{n},
\end{align*}
where the  function $f$ is defined at the beginning of this section.
We write \[\rho(r)=\sup_{x\in[0,r]}x\kappa^{-}(x),\] 
then by \eqref{eq:f} we have
\begin{align*}
 & \quad-\kappa(|\Delta_{t}^{n}|)|\Delta_{t}^{n}|f^{\prime}(|\Delta_{t}^{n}|)+2f^{\prime\prime}(|\Delta_{t}^{n}|)\vert\lambda^{n}(\Delta_{t}^{n})\vert^{2}\\
 & \le-|\Delta_{t}^{n}|\vert\lambda^{n}(\Delta_{t}^{n})\vert^{2}-\kappa(|\Delta_{t}^{n}|)|\Delta_{t}^{n}|\vert\pi^{n}(\Delta_{t}^{n})\vert^{2}\\
 & \le-|\Delta_{t}^{n}|+\rho(\delta_{n})+\delta_{n}.
\end{align*}
It follows from taking the expectation  that
\begin{align*}
\frac{\md}{\md t}\bE f(|\Delta_{t}^{n}|) & \le-\bE|\Delta_{t}^{n}|+ \rho(\delta_{n})+\delta_{n}+\eta f'(0)(\mathcal{K}_{n-1}(\E|\Delta_{t}^{i}|)+ C\theta_{n}^{\frac{1}{d+2}}).\label{eq:501}
\end{align*}
Since \(\lim_{r\to 0}\rho(r)=0\), we take \(\delta_{n}\) such that \(\delta_{n}+\rho(\delta_{n})\leq \theta_{n}^{\frac{1}{d+2}}\) and \(0<\varepsilon< f'(0)^{-1}-\eta\).
Thus, 
\begin{align*}
\frac{\md}{\md t}\bE f(|\Delta_{t}^{n}|) & \le-\bE|\Delta_{t}^{n}|+\eta f'(0)\mathcal{K}_{n-1}(\E|\Delta_{t}^{i}|)+ C\theta_{n}^{\frac{1}{d+2}}\\
&\leq -\varepsilon \E f(|\Delta_{t}^{n}|)-(1-\varepsilon f'(0))\E|\Delta_{t}^{n}| +\eta f'(0)\mathcal{K}_{n-1}(\E|\Delta_{t}^{i}|)+ C\theta_{n}^{\frac{1}{d+2}}.
\end{align*}
By Gronwall's inequality, we get 
\begin{equation}
    \label{eqn-uniform}
0\leq \E f(|\Delta^{n}_{t}|)\leq \me^{-\varepsilon t} \int_{0}^{t}\me^{\varepsilon s}[(-1+\varepsilon f'(0))\E|\Delta_{s}^{n}|+\eta f'(0) \mathcal{K}_{n-1}(\E|\Delta_{s}^{i}|)]\md s +C\theta_{n}^{\frac{1}{d+2}}.
\end{equation}
This implies 
\begin{equation*}
    \int_{0}^{t}\me^{s}\E|\Delta_{s}^{n}|\md s\leq \frac{\eta f'(0)}{1-\varepsilon f'(0)} \int_{0}^{t}\me^{\varepsilon s}\cK_{n-1}(\E|\Delta_{s}^{i}|)\md s + C\me^{\varepsilon t}\theta_{n}^{\frac{1}{d+2}}.
\end{equation*}
We denote 
\begin{equation*}
    s_{n}=\int_{0}^{t}\me^{\varepsilon s}\cK_{n}(\E|\Delta_{s}^{i}|)\md s \quad\text{and}\quad \delta= 1-\frac{\eta f'(0)}{1-\varepsilon f'(0)}\in(0,1).
\end{equation*}
Then, we have 
\begin{align*}
    s_{n}&=s_{n-1}+\alpha_{n}\Bigl( \int_{0}^{t}\me^{s}\E|\Delta_{s}^{n}|\md s-s_{n-1}\Bigr)\\
    &\leq (1-\delta \alpha_{n-1})s_{n-1}+\alpha_{n}C\me^{\varepsilon t}\theta_{n}^{\frac{1}{d+2}}.
\end{align*}

If \(\overline{\alpha}<2-\alpha_{\infty}\), then by Corollary \ref{cor:weight} we have \(\theta_{n}\leq C \alpha_{n}\).
We take \(B_{n}=C\me^{\varepsilon t}\alpha_{n}^{\frac{1}{d+2}}\), and calculate
\begin{equation*}
    \beta_{n}=  \limsup_{n\to\infty}\frac{\alpha_{n-1}^{1/(d+2)}-\alpha_{n}^{1/(d+2)}}{ \alpha_{n}^{1/(d+2)}}.
\end{equation*}
Then we have
\begin{equation*}
    \limsup_{n\to\infty}\frac{\alpha_{n-1}^{1/(d+2)}-\alpha_{n}^{1/(d+2)}}{ \alpha_{n}^{(d+3)/(d+2)}}=\frac{\overline{\alpha}}{d+2}<1,
\end{equation*}
and by Lemma \ref{lem-gw} we derive \(s_{n}\leq C \me^{\varepsilon t}\alpha_{n}^{\frac{1}{d+2}}\).

In the case of \(\overline{\alpha}\geq 2\), for any \(0<\lambda<2\wedge \underline{\alpha}\), we have 
\begin{equation*}
    \theta_{n}\leq C\prod_{i=1}^{n}(1-\lambda \alpha_{i}\mathbbm{1}_{\{\lambda \alpha_{i}<1\}}).
\end{equation*} 
Set \(B_{n}=\prod_{i=1}^{n-1}(1- \lambda\alpha_{i}\mathbbm{1}_{\{\lambda\alpha_{i}<1\}})^{1/(d+2)}\), then we derive 
\begin{equation*}
    \beta_{n}=\frac{1}{\delta}[1-(1-\lambda \alpha_{n-1}\mathbbm{1}_{\{\lambda\alpha_{n-1}<1\}})^{1/(d+2)}].
\end{equation*}
Thus, we compute 
\begin{equation*}
    \lim_{n\to\infty}\frac{\beta_{n+1}}{\alpha_{n}}=\lim_{n\to\infty} \frac{1-(1-\lambda\alpha_{n})^{\frac{1}{d+2}}}{\delta\alpha_{n}}=\frac{\lambda}{\delta(d+2)}.
\end{equation*}
By Lemma \ref{lem-gw}, we obtain, for any \(\delta<1\wedge\frac{\underline{\alpha}}{d+2} \wedge \frac{2}{d+2}\),
\begin{equation*}
    s_{n}\leq C\me^{\varepsilon t}\prod_{i=1}^{n}(1-\delta\alpha_{i}).
\end{equation*}

Plugging above estimates into \eqref{eqn-uniform}, we get the uniform estimates of \(\E f(|\Delta_{t}^{n}|)\) as well as \(\E[\cW_{f}(\cL(X_{t}^{n}),\mu_{t})]\).
The estimate of \(\E[\cW_{f}(\mu_{t}^{n},\mu_{t})]\) follows from the triangle inequality 
\begin{align*}
    \cW_{f}(\widehat{\mu}_{t}^{n},\mu_{t})&\leq \cW_{f}(\widehat{\mu}_{t}^{n},\cK_{n}(\delta_{\widehat{Y}_{t}^{i}})) +\cW_{f}(\cK_{n}(\delta_{\widehat{Y}_{t}^{i}}),\mu_{t})\\
    &\leq \cK_{n}(f(\Delta_{t}^{n}|))+f'(0)\cW_{1}(\cK_{n}(\delta_{\widehat{Y}_{t}^{i}}),\mu_{t}).
\end{align*}
To sum up, we  obtain:
\begin{enumerate}
    \item if \(\overline{\alpha}<2-\alpha_{\infty} \), we have \[\sup_{t\geq 0}\E[\cW_{f}(\mu_{t}^{n},\mu_{t})+\cW_{f}(\cL(X^{n}_{t}),\mu_{t})]\leq C\alpha_{n}^{\frac{1}{d+2}};\]
    \item if \(\overline{\alpha}\geq 2\), we have, for any \(\delta<1\wedge\frac{\underline{\alpha}}{d+2} \wedge \frac{2}{d+2}\),
    \[\sup_{t\geq 0}\E[\cW_{f}(\mu_{t}^{n},\mu_{t})+\cW_{f}(\cL(X^{n}_{t}),\mu_{t})]\leq C\prod_{i=1}^{n}(1-\delta\alpha_{i}).\]
\end{enumerate}

\begin{rem}
\label{rem:contrac}Let $X,Y$ be solutions of \eqref{eq:McKean--Vlasov-1}
with different initial data. 
Then by a similar argument on the reflection coupling \((\widehat{X},\widehat{Y})\), one can obtain that
\begin{align*}
\frac{\md}{\md t}\bE f(|\widehat{X}_{t}-\widehat{Y}_{t}|) & \le-\bE|\widehat{X}_{t}-\widehat{Y}_{t}|+c^{-1}\eta\cW_{1}(\cL(\widehat{X}_{t}),\cL(\widehat{Y}_{t}))\\
 & \le-(1-c^{-1}\eta)\bE|\widehat{X}_{t}-\widehat{Y}_{t}|\le-(c-\eta)\bE f(|\widehat{X}_{t}-\widehat{Y}_{t}|),
\end{align*}
which implies that
\[
\cW_{f}(\cL(X_{t}),\cL(Y_{t}))\le\me^{-(c-\eta)t}\cW_{f}(\cL(X_{0}),\cL(Y_{0})),
\]
where $c=1/f'(0)$.
\end{rem}

\section{Numerical experiments\label{sec:6}}

In this section, we shall illustrate our results of the SPoC algorithm with different numerical examples. 
We discretize the time interval $[0,T]$ into \(\{t_{0},\dots,t_{M}\}\), where  $t_{m}=m\Delta t$ and \(\Delta t= T/M\).
We give the pseudocode in Algorithm \ref{algorithm}. 

\begin{algorithm}[htbp]
  \caption{\label{algorithm}Framework of SPoC algorithm} 
  \begin{algorithmic}[1]
    \Require
      The initial distribution \(\mu_{0}\), the update rate \(\{\alpha_{n}\}\);
    \State Initialize \(n=1\), $X_{0}^{1}\sim\mu_{0}$;
    \For{\(m=0\) \textbf{to} \(M\)}
    \State \(\mu_{t_{m}}^{1}=\delta_{X_{0}^{1}}\);
    \EndFor
    \Repeat
        \State \(n=n+1\);
        \State \(X_{0}^{n}\sim \mu_{0}\);
        \State \(\mu_{0}^{n}=\mu_{0}^{n-1}+\alpha_{n}(\delta_{X_{0}^{n}}-\mu_{0}^{n-1})\);
         \For{$m = 1$ \textbf{to} \(M\)}
         \State Generate a path of $X^{n}$ by 

        \State $\quad\quad$ $X_{t_{m}}^{n}=X_{t_{m-1}}^{n}+b(t,X_{t_{m-1}}^{n},\mu_{t_{m-1}}^{n-1})\Delta t+\sigma(t,X_{t_{m-1}}^{n},\mu_{t_{m-1}}^{n-1})\Delta W_{t}^{n};$

         \State Update the empirical measure by

        \State $\quad\quad$ $\mu_{t_{m}}^{n}=\mu_{t_{m}}^{n-1}+\alpha_{n}(\delta_{X_{t_{m}}^{n}}-\mu_{t_{m}}^{n-1});$
        \State $m=m+1$;
      
     \EndFor
         \Until{the end condition.}
  \end{algorithmic}
\end{algorithm}

\begin{algorithm}[htbp]
  \caption{\label{algorithm2}Batch-by-batch SPoC algorithm} 
  \begin{algorithmic}[1]
    \Require
    The initial distribution \(\mu_{0}\), the update rate \(\{\alpha_{n}\}\), the bath size \(\{N_{n}\}\);
   
       \State Initialize \(n=1\), and the~first~batch~$X_{0}^{1,1},...,X_{0}^{N_{1},1}\ensuremath{\stackrel{\mathrm{i.i.d.}}{\sim}}\mu_{0}$;
       \For{\(m=0\) \textbf{to} \(M\)}
       \State \(\mu_{t_{m}}=\frac{1}{N_{1}}\sum_{i=1}^{N_{1}}\delta_{X_{0}^{i,1}}\);
        \EndFor
       \Repeat
       \State \(n=n+1\);
        \State Sample  $X_{0}^{1,n},...,X_{0}^{N_{n},n}\ensuremath{\stackrel{\mathrm{i.i.d.}}{\sim}}\mu_{0}$;
         \For{$m = 1$ \textbf{to} $M$}
         \For{$i = 1$ \textbf{to} $N_n$}
         \State Generate a path of  $X^{i,n}$ by 
        \State $\quad\quad$ $X_{t_{m}}^{i,n}=X_{t_{m-1}}^{i,n}+b(t,X_{t_{m-1}}^{i,n},\mu_{t_{m-1}}^{n})\Delta t+\sigma(t,X_{t_{m-1}}^{i,n},\mu_{t_{m-1}}^{n-1})\Delta W_{t}^{i,n};$
             \EndFor
      \State Update the empirical measure by
        \State $\quad\quad$ $\mu_{t_{m}}^{n}=\mu_{t_{m}}^{n-1}+\alpha_{n}(\frac{1}{N_n}\sum_{i=1}^{N_{n}}\delta_{X_{t_{m}}^{i,n}}-\mu_{t_{m}}^{n-1});$
     \EndFor
    \Until{the end condition.}
  \end{algorithmic}
\end{algorithm}

For the SPoC algorithm, the total computational complexity of the first $N$ particles is $\mathcal{O}(N^{2}M)$. 
For the $n$-th particle, it interacts with all the previous particles.
At each time step \(t_{m}\), it takes \(n\) calculations to update the empirical measure \(\mu_{t_{m}}^{n}\).
This accounts for a total \(\mathcal{O}(nM)\) complexity of the \(n\)-th particle.
Thus, the total computational complexity of the first $N$ particles is $\mathcal{O}(N^{2}M)$. 
In particular, if the interaction is in the form of \(b(t,X_{t},\E X_{t})\), then the complexity reduces to $\mathcal{O}(NM)$.
It is because we only need to spend $\mathcal{O}(1)$ time to update $\bE X_{t_{m}}$ for each particle.
More generally, we can add the particles batch-by-batch rather than
particle-by-particle: we add $N_{k}$ particles $\{X^{i,k}:i=1,\dots,N_{k}\}$
in the $k$-th batch, and update the empirical measure from the batch. 
The batch model could be more efficient in practical, see Algorithm \ref{algorithm2} for the pseudocode.

We conclude with a comprehensive discussion on the differences between the SPoC algorithm and the classical PoC algorithm. 
For the classical PoC algorithm, we decide the number of the particles
$N$ beforehand and simulate all the paths of $N$ particles $X^{1,N},\dots,X^{N,N}$ simultaneously.

In order to get the empirical measure corresponding to the terminal time $T$, we have to calculate the whole system until the last time step. 
While in the SPoC algorithm, we simulate the whole path of a single particle at once, and then sequentially increase the number of particles.
Given the paths of \(X^{1},\dots,X^{N-1}\), we can generate the path of \(X^{N}\). 
We have already obtained an approximate distribution $\mu_{T}^{n}=\frac{1}{n}\sum_{i=1}^{n}\delta_{X_{T}^{i}}$ at the intermediate step of the calculation. 
Moreover, as the number of simulated particles increases, the empirical measure converges to the true distribution.
While in the classical PoC algorithm, we need to recalculate the whole system to improve the approximation accuracy.
In Figure \ref{fig:Comparison} we compare the structures of these two algorithms.

\begin{figure}[htbp]
\includegraphics[width=\textwidth]{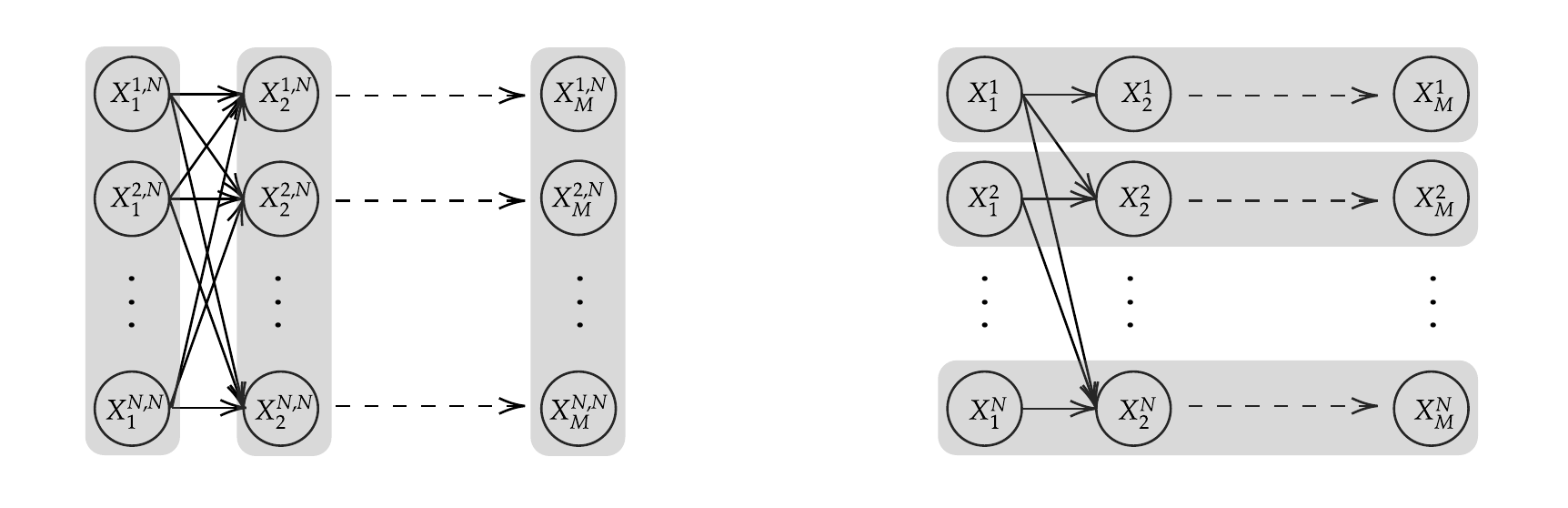}

\caption{\label{fig:Comparison}Comparison of the classical PoC (left) and the sequential PoC (right).}
\end{figure}

In the following we show the performance of the SPoC algorithm with some examples.
\subsubsection*{Example 1}

We consider the mean-field Ornstein--Uhlenbeck process:
\begin{equation}
\md X_{t}=(-2X_{t}-\bE X_{t})\md t+(2-\sqrt{\bE|X_{t}|^{2}})\md W_{t}.\label{eq:ex-ou}
\end{equation}
Both the drift and diffusion terms are distribution dependent.
And it is straightforward to check that \eqref{eq:ex-ou} has
a unique invariant measure $\mu^{*}\sim\mathcal{N}(0,4/9)$.
In our numerical experiments, we take $X_{0}^{n}=1$, $T=1$ and divide
the interval $[0,T]$ uniformly into $M=30$ parts.
We sample $10^{6}$ particles and repeat the simulation independently
$100$ times. 
We record the mean and the second moment of the empirical measures. 
In Figure \ref{fig:mean-field-Ornstein=002013Uhlenbeck},
we compare the convergence of the mean (left) and  the second
moment (right) of the SPoC and the classical PoC.
We plot the $90$\% confidence interval by the shaded area to show that the SPoC algorithm has almost the same convergence performance
as the classical PoC.

\begin{figure}[htbp]
    \includegraphics[width=0.5\textwidth]{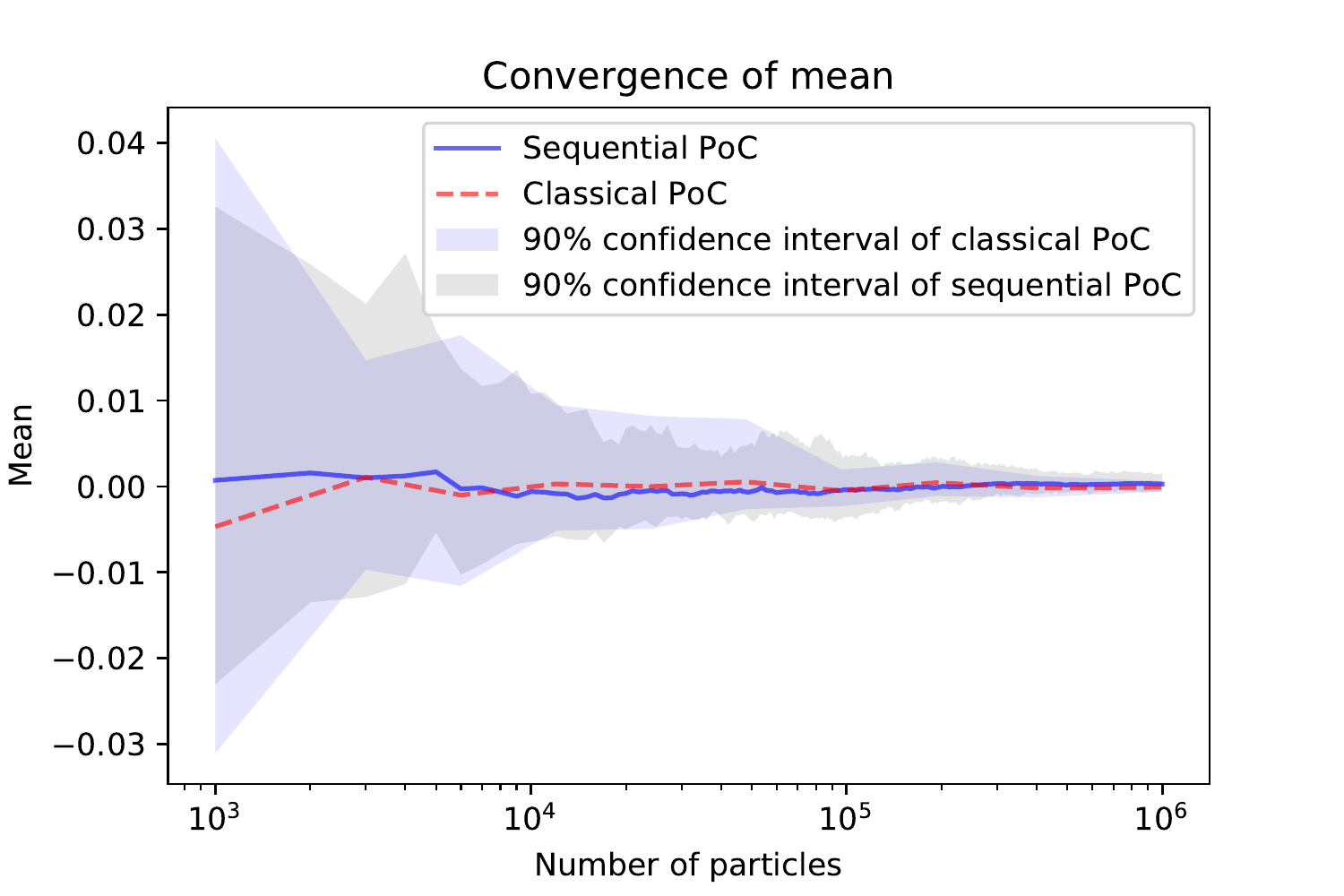}\includegraphics[width=0.5\textwidth]{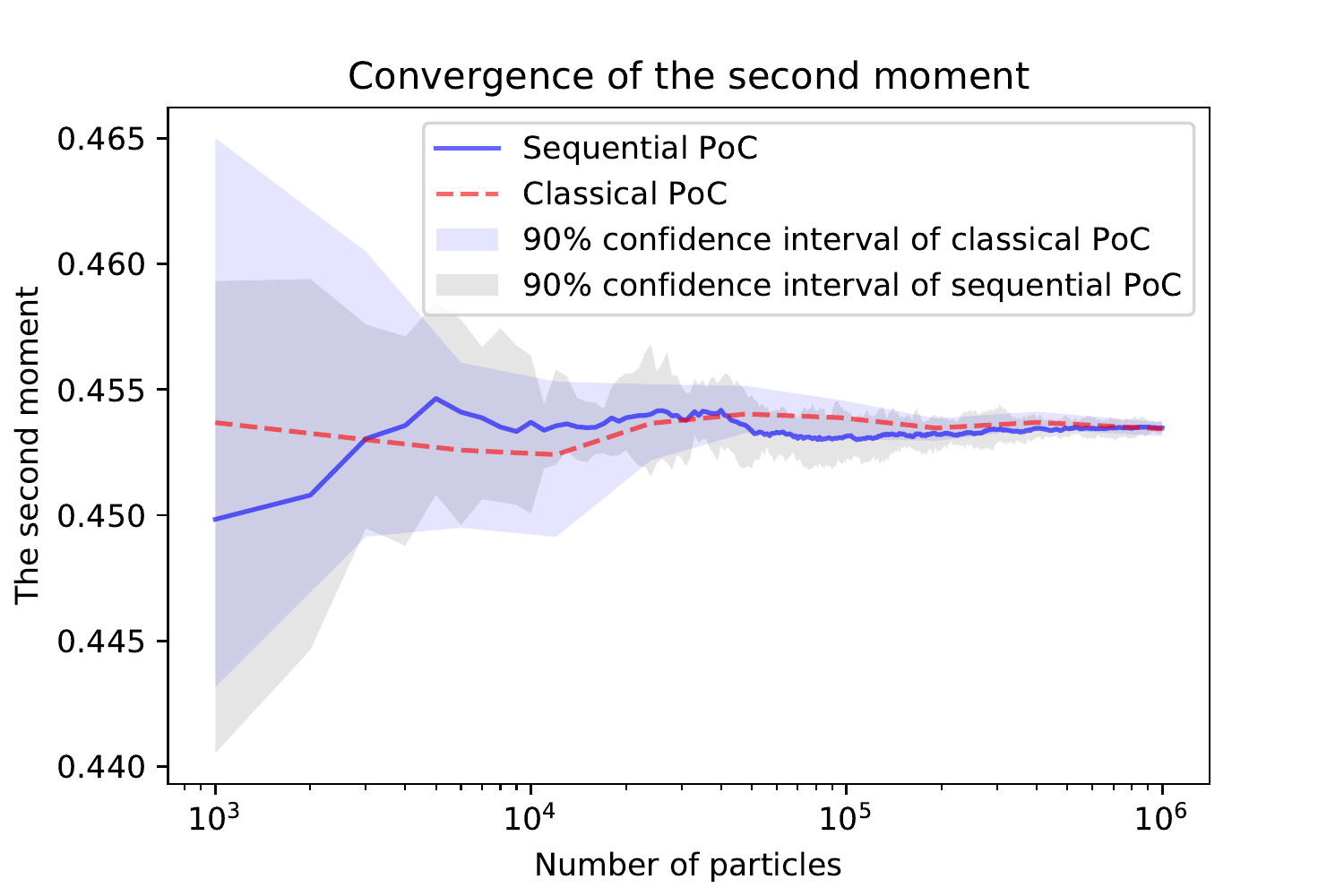}
    
    \caption{\label{fig:mean-field-Ornstein=002013Uhlenbeck}Mean-field Ornstein--Uhlenbeck
    model.}
    
    \end{figure}
    
\subsubsection*{Example 2}

Consider a multidimensional nonlinear system:
\begin{equation}
\md X_{s}=\big(-\alpha X_{s}+\frac{\vec{e}}{\beta+\Vert X_{s}-\bE X_{s}\Vert^{2}}\big)\md s+\sigma\md B_{s},\label{eq:example2}
\end{equation}
where $X_{s}\in\bR^{3}$ and $\vec{e}=\frac{X_{s}-\bE X_{s}}{\Vert X_{s}-\bE X_{s}\Vert}$
is a unit vector. 
We set $\alpha=0.1$, $\beta=0.1$, $\sigma=1$, $X_{0}\sim\mathcal{N}(0,1)$, $T=1$.
The damping term $-\alpha X_{s}$ will drag each particle toward the origin, while the interaction term is a repulsive force  pushing each particle away from the mean value. 
We divide the interval $[0,T]$ uniformly into $M=20$ parts and sample $N=10^{6}$ particles. 
Figure \ref{fig:3D-model} shows the convergence of the  mean (left) and the second moment
(right) of the SPoC and the classical PoC algorithms. 
\begin{figure}[htbp]
\includegraphics[width=0.5\textwidth]{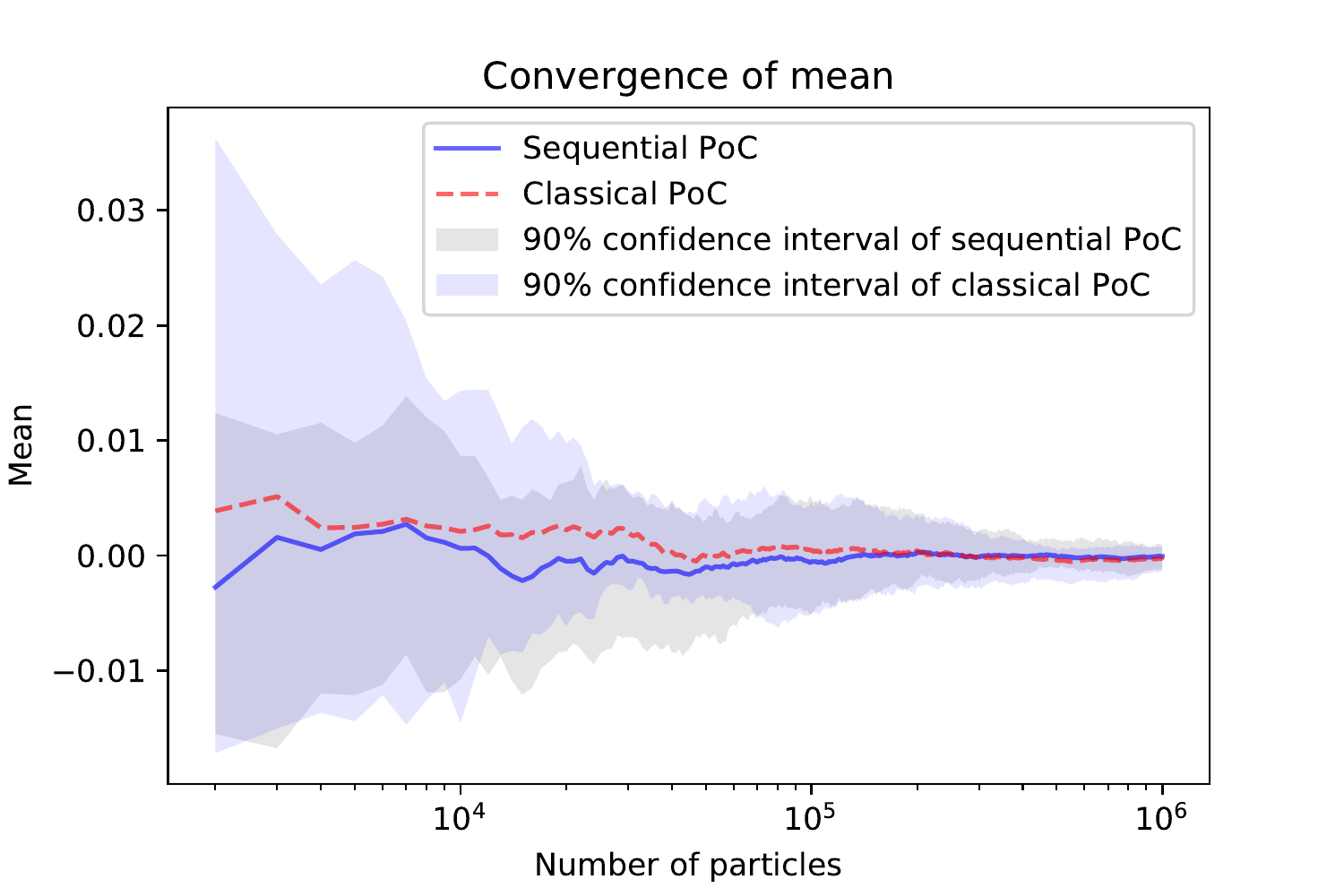}\includegraphics[width=0.5\textwidth]{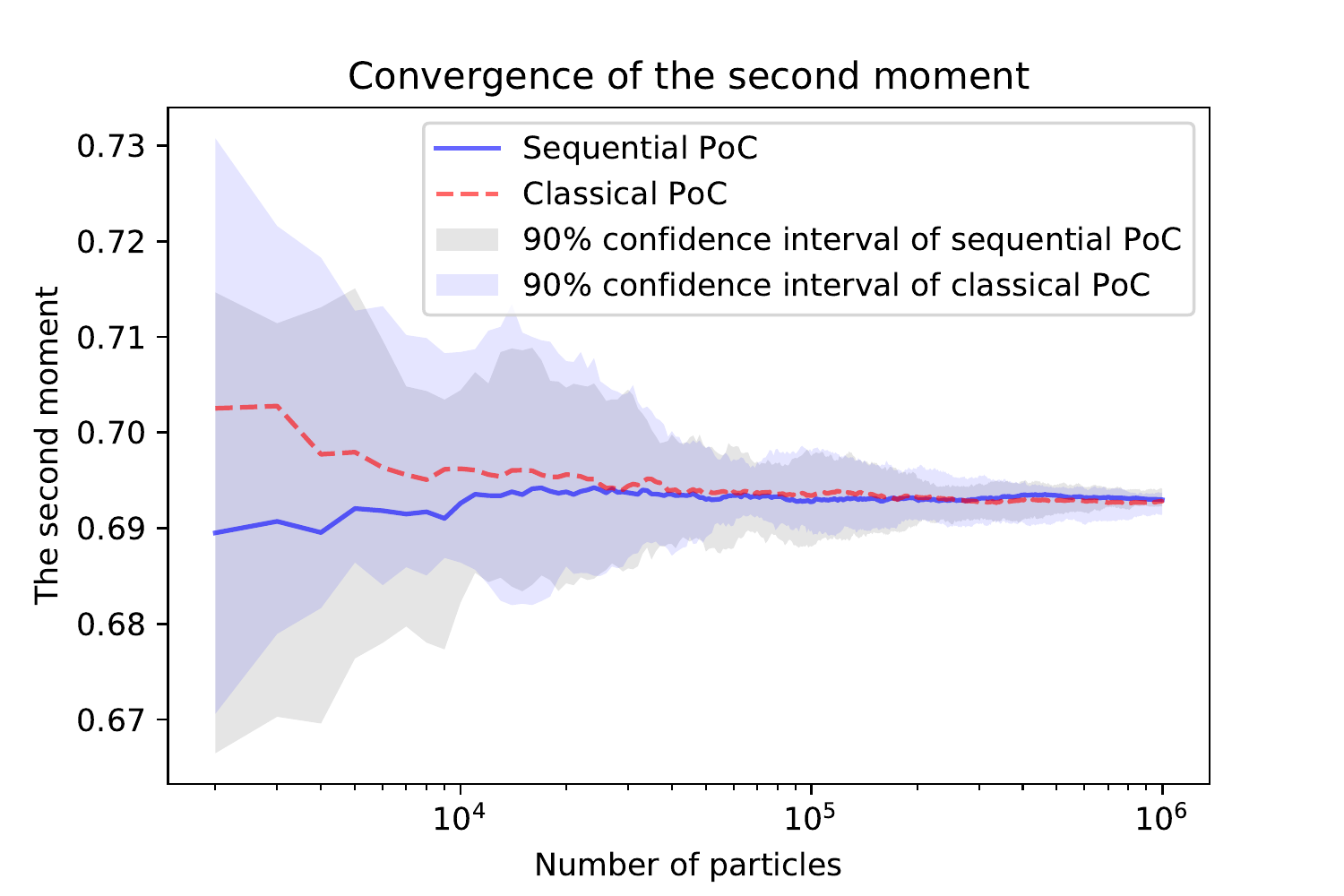}

\caption{\label{fig:3D-model} Multidimensional nonlinear model.}
\end{figure}

\subsubsection*{Example 3}
Recall the Curie--Weiss mean-field lattice model 
\[
\md X_{t}=[-\beta(X_{t}^{3}-X_{t})+\beta K\bE X_{t}]\md t+\sigma\md W_{t}.
\]
We take $\beta=1.0$, $K=0.5$, $T=1$ and $\sigma=1.0$.
Similar to the previous examples,  Figure \ref{fig:Curie-Weiss} shows the convergence of the mean and the second moment.
\begin{figure}[htbp]
\includegraphics[width=0.5\textwidth]{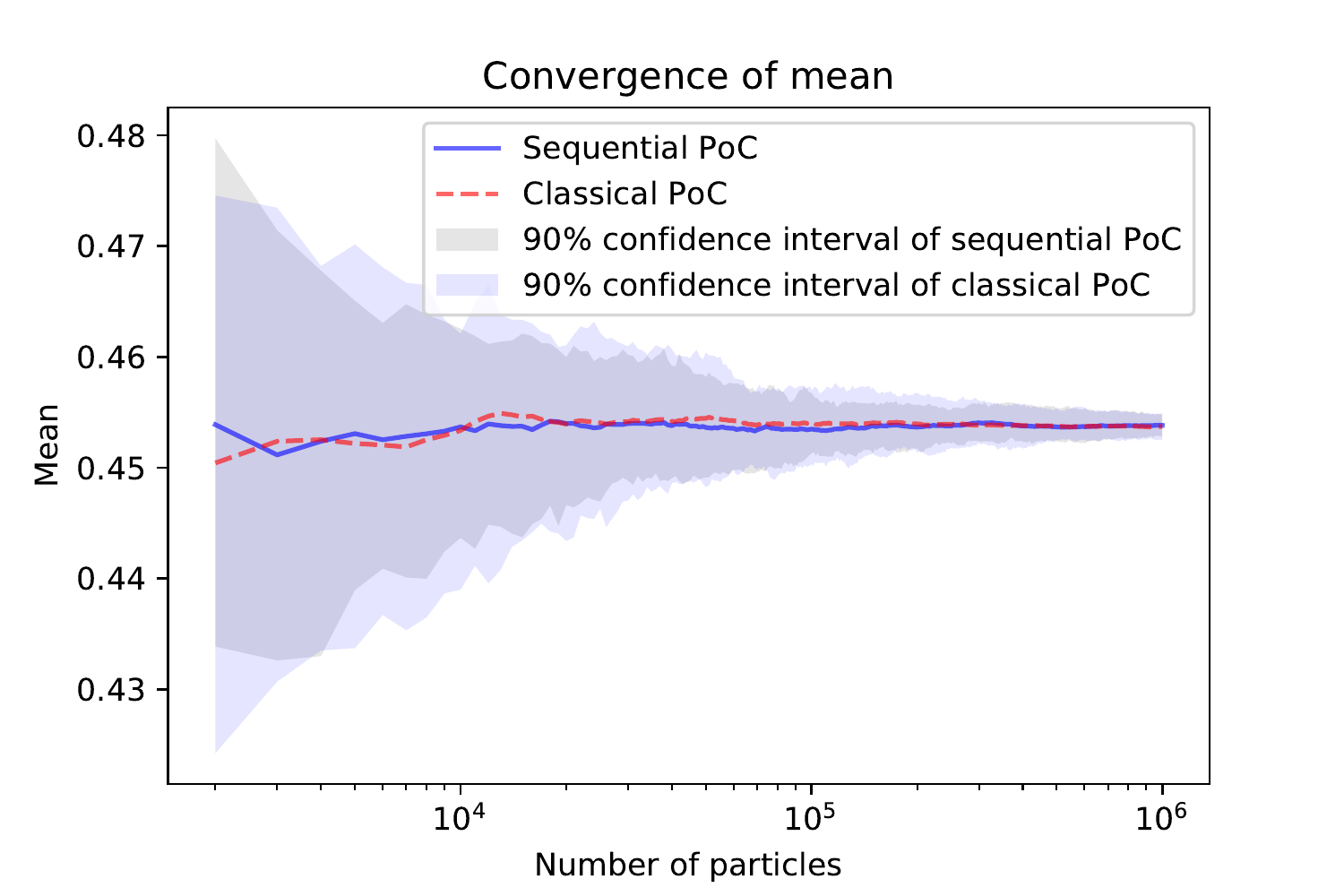}\includegraphics[width=0.5\textwidth]{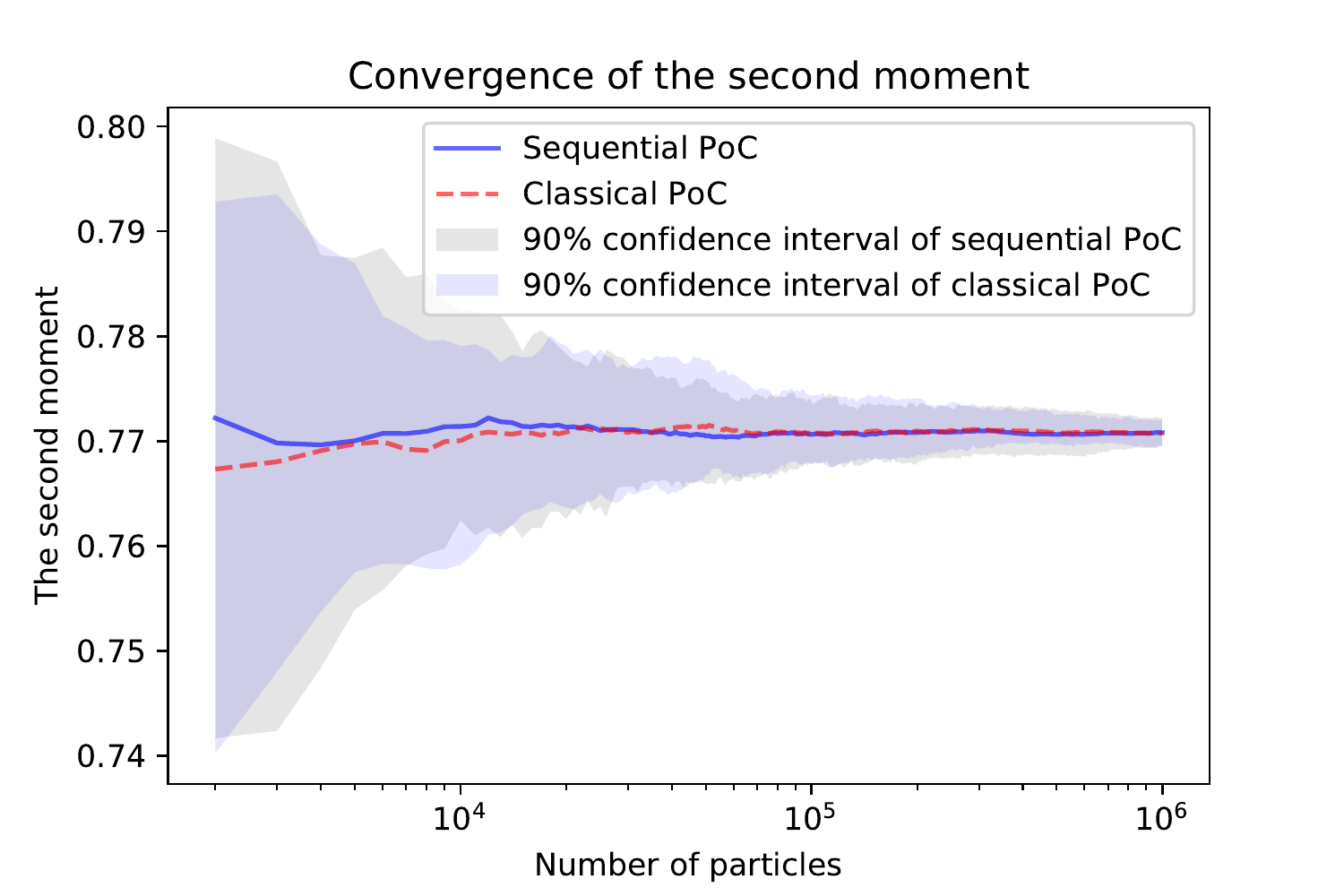}

\caption{\label{fig:Curie-Weiss}Curie--Weiss mean-field lattice model.}
\end{figure}
Furthermore, we plot the time evolution of the distribution in Figure \ref{fig:CW-pdf} to demonstrate the asymptotic behavior of the SPoC algorithm.
At $T=0.1,1,2,3,10$  we plot the approximate densities of the empirical measures with the initial condition \(X_{0}=1\) (left) and \(X_{0}=0\) (right) and compare them with the density of the invariant measure.
\begin{figure}[htbp]
\includegraphics[width=0.5\textwidth]{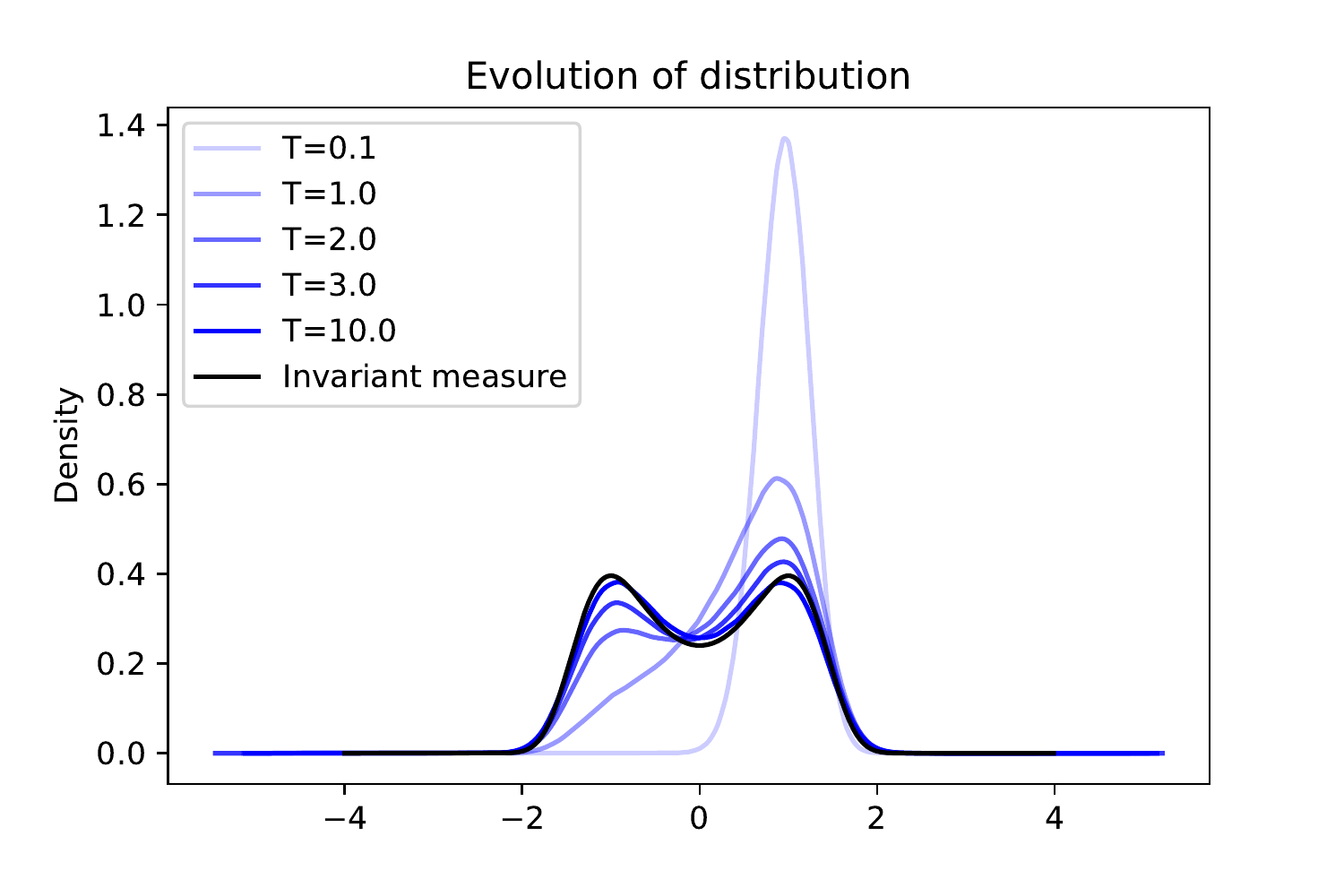}\includegraphics[width=0.5\textwidth]{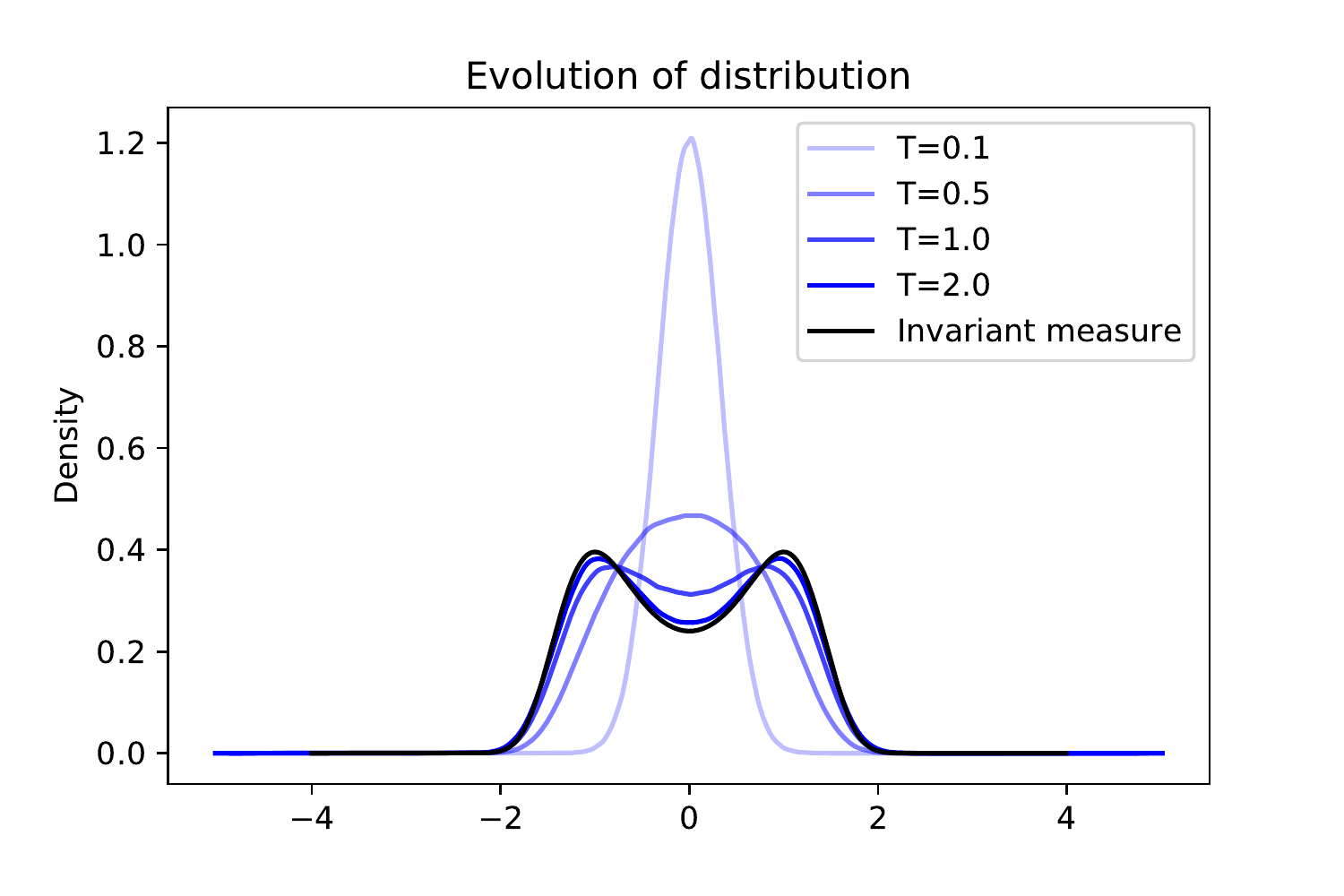}

\caption{\label{fig:CW-pdf}Approximate densities of empirical measures and the invariant measure.}
\end{figure}

\appendix

\section{Convergence of empirical measures on the path space}
\label{appendix}
In the appendix we give a rigorous discussion for Remark \ref{rem:functional}.
First, we prove an infinite-dimensional version of Lemma \ref{lemma:iid}.
Let \(E=C([0,1];\bR^{d})\) be the continuous path space equipped with the uniform norm, \(\mathscr{P}(E)\) the set of Borel probability measures.
For any \(\mu,\nu\in\mathscr{P}(E)\) their Wasserstein distance \(\cW_{2}\) is defined by 
\begin{equation*}
    \cW_{2}(\mu,\nu):=\inf\Bigl\{\Bigl(\E\int_{0}^{1}|X_{t}-Y_{t}|^{2}\md t\Bigr)^{1/2}:\cL(X)=\mu,\,\cL(Y)=\nu\Bigr\}.
\end{equation*}
Let \(\mu\in\mathscr{P}(C([0,1];\bR^{d}))\) be the law of the following It\^o diffusion
\begin{equation}
    \label{eqn-sde}
    X_{t}=X_{0}+\int_{0}^{t}b(s,X_{s})\md s +\int_{0}^{t}\sigma(s,X_{s} )\md W_{s}.
\end{equation}

\begin{assumption}
    \label{asmp-path}
    We assume \eqref{eqn-sde} has a unique strong solution \(X\) that satisfies 
    \begin{equation*}
        \E\sup_{|t-s|\leq k^{-1}}|X_{t}-X_{s}|^{2}\leq C k^{-\frac{1}{2}} \quad\text{and}\quad         \lim_{k\to \infty}  \E\sup_{|t-s|\leq k^{-1}}|X_{t}-X_{s}|^{dk+5}=0.
    \end{equation*}
\end{assumption}

\begin{lem}
    \label{lem-path}
    Suppose Assumption \ref{asmp-path} holds.
    Let \((X^{n})_{n\ge 1}\) be i.i.d. copies of  the solution $X$ of \eqref{eqn-sde}, i.e., $X^n$ satisfies
\begin{equation*}
X^{n}_{t}=X_{0}^{n}+\int_{0}^{t}b(s,X^{n}_{s})\md s +\int_{0}^{t}\sigma(s,X^{n}_{s} )\md W^{n}_{s},
\end{equation*}
where \(W^{n}\) are independent Brownian motions and \(X_{0}^{n}\) are i.i.d. random variables independent of \(\{W^{n}\}_{n\geq 1}\).
    By \(\mu^{n}\) we denote the empirical measure of \(X^{n}\), i.e.,
    \begin{equation}
        \mu^{n}:=\frac{1}{n}\sum_{i=1}^{n}\delta_{X^{i}},
        \end{equation}
    and by \(\mu\) we denote the law of \(X^{n}\).
    Then, we have 
    \begin{equation*}
        \E[\cW_{2}(\mu^{n},\mu)^{2}]\leq C\frac{\ln\ln n}{\ln n}. 
    \end{equation*}
\end{lem}

\begin{proof}
    The proof can be split into two steps: we first discretize the continuous path space and estimate the distance between measures on \(C([0,1];\bR^{d})\) through their finite marginals; then we apply the classical mollification argument to the finite marginals and obtain a uniform estimate.
    
    For the first step, we define a time partition \(\{t_{0},\dots,t_{k}\}\) of \([0,1]\) where \(t_{l}=l/k\) for \(0\leq l\leq k\).
    We  define  a map \(T_{k}:C([0,1];\bR^{d})\to C([0,1];\bR^{d})\)  for any \(t_{l}\leq t\leq t_{l+1}\) as
\begin{equation*}
    T_{k}(\omega)_{t}:=(t_{l+1}-t)k\omega_{t_{l}}+(t-t_{l})k\omega_{t_{l+1}}.
\end{equation*}
Notice \(T_{k}(\omega)\) is a piecewise linear path and has the same value as \(\omega\) on \(\{t_{0},\dots,t_{k}\}\).
We define pushforward measures
\begin{equation*}
    \mu_{k}^{n}=(T_{k})_{\#}\mu^{n}\quad \text{  and  } \quad    \mu_{k}=(T_{k})_{\#}\mu.
\end{equation*}
Since \(X^{n}\) are i.i.d., we have 
\begin{align*}
    \E[\cW_{2}(\mu^{n},\mu_{k}^{n})^{2}]&\leq \E \Bigl[\frac{1}{n}\sum_{i=1}^{n}\int_{0}^{1}|X^{i}_{t}-T_{k}(X^{i})_{t}|^{2}\md t\Bigr]\\
    &=\E\Bigl[\int_{0}^{1}|X_{t}-T_{k}(X)_{t}|^{2}\md t\Bigr]\\
    &\leq \E\bigl[\sup\{|X(t)-X(s)|^{2}:|t-s|\leq k^{-1}\}\bigr]\\
    &\leq Ck^{-1}.
\end{align*}
Similarly, we have \(\cW_{2}(\mu,\mu_{k})^{2}\leq Ck^{-1}\).
To estimate \(\E[\cW_{2}(\mu^{n}_{k},\mu_{k})^{2}]\), let \(\xi\) and \(\zeta\) be two random processes with law \(\mu_{k}^{n}\) and \(\mu_{k}\) respectively.
Since \(\mu^{n}_{k}\) and \(\mu_{k}\) concentrate on the piecewise linear paths, we have 
\begin{equation*}
    \E\Bigl[\int_{0}^{1}|\xi_{t}-\zeta_{t}|^{2}\md t\Bigr]\leq Ck^{-1}\E\Bigl[\sum_{l=1}^{k}|\xi(t_{l})-\zeta(t_{l})|^{2}\Bigr],
\end{equation*}
which implies 
\begin{equation*}
    \E[\cW_{2}(\mu^{n}_{k},\mu_{k})^{2}]\leq Ck^{-1}\E[\cW_{2}(\bar{\mu}^{n}_{k},\bar{\mu}_{k})^{2}],
\end{equation*}
where \(\bar{\mu}^{n}_{k},\bar{\mu}_{k}\in\sP(\bR^{d\times k})\) are the marginals of  \(\mu^{n}_{k},\mu_{k}\) on \(\{t_{1},\dots,t_{k} \}\) respectively.
To sum up, we derive 
\begin{equation}
    \label{eqn-path1}
    \begin{aligned}
    &\quad\E[\cW_{2}(\mu_{n},\mu)^{2}]\\
    &\leq \E[\cW_{2}(\mu^{n},\mu^{n}_{k})^{2}]+\E[\cW_{2}(\mu^{n}_{k},\mu_{k})^{2}]+\E[\cW_{2}(\mu_{k},\mu)^{2}]\\
    &\leq Ck^{-1}+Ck^{-1}\E[\cW_{2}(\bar{\mu}^{n}_{k},\bar{\mu}_{k})^{2}].
    \end{aligned}
\end{equation}

We proceed to the mollification step.
Let \(\Phi_{\sigma}\) be the law of \((\sigma B_{t_{1}},\dots, \sigma B_{t_{k}})\), where \(B\) is a \(d\)-dimensional Brownian motion.
Then \(\Phi_{\sigma}\) has a density function 
\begin{equation*}
    \phi_{\sigma/\sqrt{k}}(x_{1})\phi_{\sigma/\sqrt{k}}(x_{2}-x_{1})\cdots\phi_{\sigma/\sqrt{k}}(x_{k}-x_{k-1}),
\end{equation*}
where \(\phi_{\sigma/\sqrt{k}}\) is the density function of \(\mathcal{N}(0,\sigma^{2}k^{-1}I_{d})\).
The mollified measures are given by
\begin{equation*}
   \mu_{k,\sigma}^{n}=\bar{\mu}_{k}^{n}*\Phi_{\sigma} \quad \text{and }\quad \mu_{k,\sigma}=\bar{\mu}_{k}*\Phi_{\sigma}
\end{equation*}
with density functions
\begin{equation*}
    g(x_{1},\dots,x_{k})=\cK_{n}(\phi_{\sigma/\sqrt{k}}(x_{1}-X_{t_{1}}^{i}))\cdots \cK_{n}(\phi_{\sigma/\sqrt{k}}(x_{k}-x_{k-1}- X^{i}_{k}+X^{i}_{k-1})),
\end{equation*}
and
\(
    f(x_{1},\dots,x_{k})=\E [g(x_{1},\dots,x_{k})].
\)
It is direct to show that
\begin{equation}
    \label{eqn-path2}
    \E [\cW_{2}(\mu^{n}_{k,\sigma},\mu^{n}_{k})^{2}]\leq C\sigma^{2}k\quad\text{and}\quad \E [\cW_{2}(\mu_{k,\sigma},\mu_{k})^{2}]\leq C\sigma^{2}k.
\end{equation}
By Lemma \ref{lem:coupling}, we get
\begin{equation}\label{eqn-path}
\begin{aligned}
    &\quad\cW_{2}(\mu_{k,\sigma},\mu_{k,\sigma}^{n})^{2}\\
    &\leq  Ck\int_{\bR^{d\times k}}\sup_{1\leq l\leq k}|x_{l}|^{2}|f(x_{1},\dots,x_{k})-g(x_{1},\dots,x_{k})|\md x_{1}\cdots \md x_{k}\\
    &\leq Ck \sqrt{\int_{\bR^{d\times k}} (1+\sup_{1\leq l\leq k}|x_{l}|^{dk+1})^{-1}\md x} \sqrt{\int_{\bR^{d\times k}} (1+\sup_{1\leq l\leq k} |x_{l}|^{dk+5})|f-g|^{2}\md x}\\
    &\leq Ck A_{k}^{1/2}\sqrt{\int_{\bR^{d\times k}} (1+\sup_{1\leq l\leq k} |x_{l}|^{dk+5})|f(x)-g(x)|^{2}\md x},
\end{aligned}
\end{equation}
where \(A_{k}\) is the surface area of \(\{x=(x_{1},\dots,x_{k})\in \bR^{d\times k}:\sup_{l}|x_{l}|=1\}\) bounded by
\begin{equation*}
    A_{k}=\int_{\sup_{1\leq l\leq k}|x_{l}|=1}\md \sigma\leq C^{k}.
\end{equation*}
Notice for independent random variables \((\xi_{l})_{1\leq l\leq k}\), we have 
\begin{equation*}
    \mathrm{Var}(\xi_{1}\cdots \xi_{k})\leq \sum_{l=1}^{k}\Bigl[\mathrm{Var}(\xi_{l})\prod_{m\neq l}\E \xi_{m}^{2}\Bigr].
\end{equation*}
For the simplicity, we write 
\begin{equation*}
    x_{1}=y_{1},\quad x_{2}=y_{1}+y_{2},\quad\dots,
    \quad x_{k}=y_{1}+\cdots+ y_{k},
\end{equation*}
and 
\begin{equation*}
     X_{1}^{i}=Y_{1}^{i},\quad X_{2}^{i}=Y_{1}^{i}+Y_{2}^{i},\quad\dots,
     \quad X_{k}^{i}=Y_{1}^{i}+\cdots+ Y_{k}^{i}.
\end{equation*}
Therefore, we derive 
\begin{align*}
    &\quad \E|f(x)-g(x)|^{2}=\mathrm{Var}(g)\\
    &\leq \sum_{l=1}^{k}\Bigl[\mathrm{Var}(\cK_{n}(\phi_{\sigma/\sqrt{k}}(y_{l}-Y_{l}^{i})))\prod_{m\neq l}\E\big|\cK_{n}(\phi_{\sigma/\sqrt{k}}(y_{m}-Y_{m}^{i}))\big|^{2}\Bigr]\\
    &\leq kn^{-1}\prod_{l=1}^{k}\E\big|\cK_{n}(\phi_{\sigma/\sqrt{k}}(y_{l}-Y_{l}^{i}))\big|^{2}\\
    &\leq kn^{-1}\prod_{l=1}^{k}\E|\phi_{\sigma/\sqrt{k}}(y_{l}-Y_{l})|^{2}
\end{align*}
Plug the above estimate into \eqref{eqn-path}, we obtain
\begin{align*}
    \E[\cW_{2}(\mu_{k,\sigma},\mu_{k,\sigma}^{n})^{2}] &\leq C^{k}\sqrt{\int_{\bR^{d\times k}}(1+\sup_{1\leq l\leq k}|y_{l}|^{dk+5})\E |f(x)-g(x)|^{2} \md y}\\
    &\leq C^{k}\sqrt{\int_{\bR^{d\times k}}(1+\sup_{1\leq l\leq k}|y_{l}|^{dk+5})kn^{-1}\prod_{l=1}^{k}\E|\phi_{\sigma/\sqrt{k}}(y_{l}-Y_{l})|^{2}\md y}\\
    &\leq C^{k}n^{-1/2}\sqrt{\int_{\bR^{d\times k}}(1+\sup_{1\leq l\leq k}|y_{l}|^{dk+5})\prod_{l=1}^{k}\E|\phi_{\sigma/\sqrt{k}}(y_{l}-Y_{l})|^{2}\md y}.
\end{align*}
Then we calculate the integral inside the square root as follows,
\begin{align*}
    &\quad\int_{\bR^{d\times k}}(1+\sup_{1\leq l\leq k}|y_{l}|^{dk+5})\prod_{l=1}^{k}\E|\phi_{\sigma/\sqrt{k}}(y_{l}-Y_{l})|^{2}\md y\\
    &=\E\int_{\bR^{d\times k}}(1+\sup_{1\leq l\leq k}|y_{l}+Y_{l}|^{dk+5})\prod_{l=1}^{k}|\phi_{\sigma/\sqrt{k}}(y_{l})|^{2}\md y\\
    &\leq 2^{dk+4}\E \int_{\bR^{d\times k}}\bigl[1+\sup_{1\leq l\leq k}(|y_{l}|^{dk+5}+|Y_{l}|^{dk+5})\bigr]\prod_{l=1}^{k}|\phi_{\sigma/\sqrt{k}}(y_{l})|^{2}\md y\\
    &=C \pi^{-\frac{dk}{2}}\sigma^{-dk}k^{\frac{dk}{2}}\E \int_{\bR^{d\times k}}\bigl[1+\sup_{1\leq l\leq k}(|y_{l}|^{dk+5}+|Y_{l}|^{dk+5})\bigr]\prod_{l=1}^{k}\phi_{\sigma/\sqrt{2k}}(y_{l})\md y\\
    &\leq C \pi^{-\frac{dk}{2}}\sigma^{-dk}k^{\frac{dk}{2}} \Bigl[1+\E\sup_{1\leq l\leq k}|Y_{l}|^{dk+5}+ \int_{\bR^{d\times k}}\sup_{1\leq l\leq k}|y_{l}|^{dk+5}\prod_{l=1}^{k}\phi_{\sigma/\sqrt{2k}}(y_{l})\md y\Bigr].
\end{align*}
We take \(\sigma^{2}=k^{-1}\), and by Assumption \ref{asmp-path} we have 
\begin{equation*}
   \lim_{k\to\infty}\Bigl[ \E\sup_{1\leq l\leq k}|Y_{l}|^{dk+5}+ \int_{\bR^{d\times k}}\sup_{1\leq l\leq k}|y_{l}|^{dk+5}\prod_{l=1}^{k}\phi_{\sigma/\sqrt{2k}}(y_{l})\md y\Bigr]=0.
\end{equation*}
This implies 
\begin{equation*}
    \E[\cW_{2}(\mu_{k,\sigma},\mu_{k,\sigma}^{n})^{2}]\leq (Ck)^{Ck}n^{-\frac{1}{2}}.
\end{equation*}
Combining this with \eqref{eqn-path1} and \eqref{eqn-path2}, we derive 
\begin{equation*}
    \E [\cW_{2}(\mu,\mu^{n})^{2}]\leq C(k^{-1}+(Ck)^{Ck}n^{-1/2}).
\end{equation*}
By taking \(k\sim \frac{\ln n}{\ln\ln n}\), we finally derive
\begin{equation*}
        \E [\cW_{2}(\mu,\mu^{n})^{2}]\leq C \frac{\ln\ln n}{\ln n}. 
\end{equation*}
This concludes the proof.
\end{proof}

\begin{cor}\label{cor-path}
    We take \(\alpha_{n}=n^{-1}\).
Let $(X^{n},\mu^{n})$  be the solutions of \eqref{eq:ADPS} and \(\mu_{t}=\cL(X_t)\) given by~\eqref{eq:McKean--Vlasov} with the initial distribution~\(\mu_{0}\).
Under Assumptions \ref{assu:lipschitz} and \ref{asmp-path}, we have 
  \begin{equation*}
    \E[\cW_{2}(\mu^{n},\mu)^{2}]\leq C \frac{\ln\ln n}{\ln n}.
  \end{equation*}
\end{cor}
\begin{proof}
    Let \(Y^{n}\) be the associated synchronous coupling.
    From the proof of Theorem \ref{thm:APoC-W}, we know 
    \begin{equation*}
         \frac{1}{n}\sum_{i=1}^{n}\E\int_{0}^{T}|X^{i}_{t}-Y^{i}_{t}|^{2}\md t\leq C n^{-\frac{1}{d+2}}.
    \end{equation*}
    Let \(\nu^{n}\) denote the empirical measure of \(Y^{n}\), i.e.,
    \begin{equation*}
        \nu^{n}=\frac{1}{n}\sum_{i=1}^{n}\delta_{Y^{i}}.
    \end{equation*}
    By Lemma \ref{lem-path}, we have 
    \begin{equation*}
        \E[\cW_{2}(\nu^{n},\mu)^{2}]\leq C \frac{\ln \ln n}{\ln n}.
    \end{equation*}
    Combing above two estimates, we derive 
    \begin{equation*}
        \E[\cW_{2}(\mu^{n},\mu)^{2}]\leq C\E[\cW_{2}(\mu^{n},\nu^{n})^{2}]+C\E[\cW_{2}(\nu^{n},\mu)^{2}]\leq C  \frac{\ln \ln n}{\ln n}.
    \end{equation*}
    The proof is complete.
\end{proof}

\bigskip
\section*{Acknowledgements}
    K.~Du was partially supported by National Natural Science
    Foundation of China (12222103), and by National Key R{\&}D Program of China (2018YFA0703900).
    Y.~Jiang was supported by the EPSRC Centre for Doctoral Training in Mathematics of Random Systems: Analysis, Modelling, and Simulation (EP/S023925/1).

\bibliographystyle{plainnat}
\bibliography{ref}

\end{document}